\documentclass[english,reqno]{amsart}

\usepackage{amsmath, appendix, ulem}
\usepackage[nobysame]{amsrefs}
\usepackage{amssymb, color}
\usepackage[margin=1in]{geometry}

\usepackage{mathrsfs}
\usepackage{graphicx}
\usepackage{subfig}
\usepackage{float}
\usepackage{epsf}
\usepackage{hyperref}
\usepackage{titletoc}

\graphicspath{{../Figures/}}

\numberwithin{equation}{section}

\newtheorem{lemma}{Lemma}[section]
\newtheorem{theorem}{Theorem}[section]
\newtheorem{proposition}{Proposition}[section]

\newtheorem{remark}{Remark}[section]

\newcommand{\NullL} {{\rm Null} \, \mathcal{L}}

\newcommand{\BB}{\mathsf{B}}

\newcommand{\RR}{\mathbb{R}}

\newcommand{\Or}{\mathcal{O}}
\newcommand{\DD}{\mathsf{D}}

\newcommand{\II}{\mathsf{I}}
\newcommand{\Id}{\mathsf{Id}}
\newcommand{\FF}{\mathsf{F}}
\renewcommand{\AA}{\mathsf{A}}
\newcommand{\PP}{\mathsf{P}}

\newcommand{\xx}{{\boldsymbol{x}}}

\newcommand{\ee}{{\boldsymbol{e}}}

\newcommand{\ud}{\,\mathrm{d}}
\newcommand{\rd}{\mathrm{d}}
\newcommand{\veps}{\varepsilon}

\newcommand{\fE}{f^\text{E}}
\newcommand{\fO}{f^\text{O}}
\newcommand{\tE}{\mathrm{E}}
\newcommand{\tO}{\mathrm{O}}

\newcommand{\mc}[1]{\mathcal{#1}}

\DeclareMathOperator{\Span}{span}

\title[An AP method for transport equations with oscillatory scattering coefficients]{An asymptotic preserving method for transport equations with
  oscillatory scattering coefficients}

\author{Qin Li} 
\address{Mathematics Department, University of Wisconsin-Madison, 480 Lincoln Dr., Madison, WI 53705 USA.}
\email{qinli@math.wisc.edu}
\author{Jianfeng Lu}
\address{Department of Mathematics, Department of Physics and Department of Chemistry, Duke University, Box 90320, Durham, NC 27708 USA.}
\email{jianfeng@math.duke.edu}
\date{\today}

\thanks{The work of Q.L.~is supported in part by a start-up fund from
  UW-Madison and National Science Foundation under the grant
  DMS-1619778. The work of J.L. is supported in part by the National
  Science Foundation under the grant DMS-1415939.}

\begin{document}

\begin{abstract}
  We design a numerical scheme for transport equations with
  oscillatory periodic scattering coefficients. The scheme is
  asymptotic preserving in the diffusion limit as Knudsen number goes
  to zero. It also captures the homogenization limit as the length
  scale of the scattering coefficient goes to zero. The proposed
  method is based on the construction of multiscale finite element
  basis and a Galerkin projection based on the even-odd
  decomposition. The  method is analyzed in the asymptotic
  regime, as well as validated numerically.
\end{abstract}

\maketitle

\section{Introduction}

We study in this paper the linear transport equation with fast
oscillatory scattering coefficients in the fluid regime.
\begin{equation}\label{eqn}
  \varepsilon \partial_t f +  v\cdot\nabla_x f = \frac{1}{\veps} \sigma^{\delta} \mathcal{L}f\,,\qquad (t, x, v) \in [0, \infty) \times \Omega \times V\,.
\end{equation}
Here $\Omega \subset \RR^d$ is the spatial domain and $V$ is the
velocity space. For transport equation, the velocity space is
typically given by $\mathbb{S}$, the unit sphere in $\RR^d$. The
function $f(t, x, v) \geq 0$ is the distribution function which gives
the particle density on the phase space $(x,v)$. More
generally, we may use a variable $\xi$ to label certain physical state
of the particle so that $f = f(t, x, \xi)$ and the transport term
takes the form $v(\xi) \cdot \nabla_x f$ where $v(\xi)$ is the
velocity of a $\xi$-state particle. 

The linear transport equation has been extensively used to describe
dynamics of identical particles such as neutrons, photons and phonons
in an environment. The particles are free streaming (the advection
term $v \cdot \nabla_x f$ in \eqref{eqn}) unless they interact
(scatter) with the background media, modeled by the collision term on
the right hand side, $\mathcal{L}$ being the collision operator and
the amplitude $\sigma^{\delta}$ (always strictly positive), known as the scattering coefficient,
is spatially dependent. In this paper we study the case that the
scattering coefficient is highly oscillatory, with length scale
indicated by $\delta \ll 1$: for instance
$\sigma^{\delta}(x) = \sigma(x/\delta)$ with $\sigma$ periodic. The
dimensionless parameter $\varepsilon$ in the equation, known as the
Knudsen number, characterizes the ratio between the mean-free path of
the particle with the macroscopic length scale. Thus a smaller
$\varepsilon$ indicates stronger interaction between particles and the
media.

The specific form of the collision operator $\mathcal{L} f$ depends on the detailed modeling of the interaction of the particles with the media, but in general, it satisfies the following properties, as for the cases of radiative transfer equations and neutron transport equations:
\begin{itemize}
\item[(1)] The null space of $\mathcal{L}$ has dimension $1$. We
  denote $\text{Null}\mathcal{L} = \text{span}\{\mathcal{F}(v)\}$ with
  $\mc{F}$, normalized, referred as the Maxwellian, which is the equilibrium state
  of the collision operator;
\item[(2)] Boundedness: $\| \mathcal{L} \|_{L^2_{\mathcal{F}^{-1}}} \leq 1 $;
\item[(3)] Dissipativeness: $\exists\, c_0 > 0$ such that for any $f$,
  $\int_V f \mathcal{L} (f) \frac{1}{\mathcal{F}} \rd{v} \leq
  -\frac{c_0}{2}\lVert f - \rho \mathcal{F} \rVert_{L^2(\mathcal{F}^{-1})}^2
   \leq 0$,
  where $\rho = \int_V f\mathcal{F}\rd{v}$;
\item[(4)] Boundedness of the generalized inverse $\mathcal{L}^{-1}$: $\exists\, C_0 \geq 0$ such that 
  $\lVert \mathcal{L}^{-1}(h) \rVert_{L^2_{\mathcal{F}^{-1}}} \leq
  C_{0} \lVert h \rVert_{L^2_{\mathcal{F}^{-1}}}$ for all $h\perp\NullL$.
\end{itemize}
For the ease of the presentation, we in this paper study the simplest
case:
$$\mathcal{L} f = \langle f\rangle_v - f\,,$$
where $\langle\cdot\rangle_v$ stands for average with respect to the
$v$ variable (or if $\xi$ is used, the average is then taken over the
$\xi$ variable). Physically, this represents that after the collision,
the velocity of the particle becomes uniformly random in the velocity
space. It is clear that $\mc{F}$, the Maxwellian, is a constant
function in velocity variable in this case.

The transport equation with multiscale scattering coefficients
involves two small parameters: $\varepsilon$, the Knudsen number,
small in the fluid regime, which restricts the time step size of the
discretization, and $\delta$, the oscillation parameter of the
scattering coefficients, that typically requires fine spatial
discretization. Our goal is to design an algorithm that overcomes the
restrictions on discretization and captures the correct asymptotic
limit for both parameters. It will turn out that capturing the correct
asymptotic limit in zero limit of $\varepsilon$ is aligned with
designing asymptotic preserving (AP) scheme, while recovering the
correct limit in the zero limit of $\delta$ is connected to  numerical
homogenization.

\subsection{Diffusion limit of transport equations}
If we start with the  equation:
\begin{equation}
  \partial_t f + v\cdot\nabla_x f = \sigma^{\delta} \mathcal{L}f\,, 
\end{equation}
and perform the parabolic scaling, which is to set
\begin{equation}\label{eqn:scaling}
  t \to \frac{t}{\varepsilon^2}\,,\quad x\to\frac{x}{\varepsilon}\,, 
\end{equation}
we obtain \eqref{eqn} after the non-dimensionalization. It is well
known that in the zero limit of $\varepsilon$, the distribution
function stabilizes and converges to the Maxwellian, in the kernel of
the collision operator. Since in our case, the kernel of
$\mc{L}$ consists of constant functions in $v$, we could set
$f\to\rho^{\delta}(t,x)\mc{F}(\xi)$. With the standard Hilbert
expansion technique, it could be shown rigorously that $\rho^{\delta}$
solves the heat equation:
\begin{equation}\label{eqn:rho}
\partial_t\rho^{\delta} - C\nabla_x\cdot\left(\frac{1}{\sigma^{\delta}}\nabla_x\rho^{\delta}\right) = 0\,,
\end{equation}
with $C$ depending on collision kernel and the dimension of the
velocity space. In the case of isotropic collision operator with 2D
velocity space we use below, $C$ is given by $\frac{1}{2}$.  The heat
equation, therefore is termed the diffusion limit of the transport
equation. Such limit of the transport equation has been known for a
long time, and rigorously proved in \cite{Papa75} for Cauchy problem
and in \cite{BSS84} for bounded domain with well-prepared boundary and
initial data.

Capturing such asymptotic limit in numerical discretization is not
trivial. The small Knudsen number appears in front of the transport
and the collision operator, making the two terms stiff. In
computation, in order to capture accurate solution when stiff terms
present, the standard approaches would require a refined time discretization step size: $\Delta t<\varepsilon$, which leads tremendous
computational cost.

The so-called asymptotic preserving (AP) is a property of a numerical
method that is able to capture the asymptotic limit with the
discretization not refining the small scales of the problem. The
framework is designed for all types of discretization, but up to now
most progress has been limited to the time domain treatment. Spatial
domain discretization requires intricate boundary layer and interface
analysis, and only limited studies have been carried
out~\cite{LiLuSun2015,LiLuSun2015JCP,GK:10}. To relax the time
discretization requirements, the focus has been placed on obtaining
uniform stability for all CFL number. Most AP schemes that have been
designed exploit implicit treatment that enlarges the stability
region. The first such type of scheme appeared in~\cite{LMM} for the
transport equation computation and was later on summarized and defined
by Jin in~\cite{Jin99}. A vast literature followed the line and were
devoted to design AP schemes for varies kinetic equations, and for the
Boltzmann equation specifically, BGK penalization~\cite{FJ10},
exponential Wild sum~\cite{Wild,expo2,DP} and micro-macro
decomposition~\cite{Klar,BLM_MM} are the three major strategies. The
underlying idea of them all is to find solvers that employ implicit
treatment at the cost of explicit discretization. For transport
equation we refer to~\cite{LiWang} where a preconditioner is designed
for the implicit scheme to accelerate the convergence of the iterative
solution. For further discussion on asymptotic preserving schemes, we
refer to reviews~\cite{JinReview,pareschi_review} for Boltzmann
equation,~\cite{Degond-Rev13} for plasma and~\cite{HuJinLi16} for
hyperbolic type equations in general.

In this paper, to get over the difficulty placed on time domain
discretization, we follow the idea proposed in~\cite{JPT2} and utilize
the even-odd decomposition with implicit treatments. The details are
given in Section~\ref{sec:evenodd}.

\subsection{Heterogeneous media with high oscillations}
Albeit the long history of deriving the diffusion limit for the
transport equation with smooth media, the asymptotic limit in the case
of heterogeneous media is much less understood. The usual diffusion
limit requires smoothness of the scattering coefficient
$\sigma^{\delta}$, which might not hold for the highly oscillatory
media in our case.  The resonances between $\varepsilon$ and $\delta$
may lead to intricate phenomena, and depending on the scaling between
$\delta$ and $\varepsilon$, different types of limit could be
obtained. In the steady case (without $\partial_t$ term), the authors
in~\cite{AllaireBal:99} studied the spectrum of the steady state for
$\delta = \varepsilon$, and in the evolution case, Dumas and Golse
derived in~\cite{DumasGolse:00} the homogenized limit for the
transport equation with $\delta \ll 1$ but $\varepsilon = 1$; Goudon
and Mellet focused on combining the homogenization limit and the
diffusion limit with $\delta = \varepsilon\ll 1$
in~\cite{GoudonMellet:01, GoudonMellet:03}. A recent paper by Ben
Addullah etc.~\cite{AbPV:12} studied the periodic oscillatory media
for transport equation with $\delta\ll\varepsilon$, in which case a
drift-diffusion limit was obtained.

Despite the results on the analytic level, either formal or rigorous,
the corresponding numerics has barely been touched. The small
oscillatory factor $\delta$ in the media produces fast oscillations in
the solution along spatial domain, and without special treatment, the
brute-force numerical algorithm requires $\Delta x<\delta$. Compared
with the difficulty brought by the small $\varepsilon$, this
difficulty is even more severe: since a small spatial discretization
may impose further restrictions of the time step size; and at the
same time increases the memory cost of the numerical computation.

It is natural to consider borrowing ideas from numerical
homogenization for elliptic and parabolic type equations,
where the focus has been put on capturing the correct homogenization limit
as $\delta \to 0$, with the spatial discretization not resolving the
fine spatial scale. During the past two decades, mainly for elliptic/parabolic type of
equations, numerical analysts have developed a variety of schemes
achieving such goal from several aspects. Many successful algorithms
are designed, including multiscale finite element method~\cite{HW97,HWC99,EHW:00},
heterogeneous multiscale
method~\cite{EEngquist:03,EMingZhang:2005,MY06,AEEV12}, proper
orthogonal decomposition~\cite{MP14}, and harmonic
mapping~\cite{OZ07}. Related to our situation, many such works are
based on constructing localized basis
functions~\cite{BL11,Gloria06,OZ14,GE10} that captures the oscillation
of the media. The detailed algorithm vary but the main idea behind
them all is to upscale the problem and explore the low rank structure
in the solution space. 

Similar methods have not been carried out for the transport equation
to the best of our knowledge. Finite difference (the so-called $S_N$, the discrete ordinate method)
and spectral method (the so-called $P_N$) are standard for velocity
domain discretization and along spatial domain, finite volume or
upwind discrete Galerkin~\cite{GK:10,JangLiQiuXiong} is mostly used. It
is obvious that these methods, if used, require $\Delta x<\delta$. To
overcome such difficulty, the basis construction techniques from
numerical homogenization need to be employed. In this paper we look
for getting better basis functions for expanding the solution space
that have the information from the oscillatory media embedded in.

\subsection{Contributions of the current work}
In this work, we will focus on the case $\veps \ll \delta \ll 1$ with
$\varepsilon$ and $\delta$ unrelated (but both small). We leave the
study of other regimes (e.g., $\delta \ll \veps \ll 1$ or
$\veps = C \delta \ll 1$ for some constant $C$) for future works.

Our goal is to design a fast and accurate numerical scheme for the
transport equation with highly oscillatory media in the diffusion
regime. The difficulty is two-fold: the time discretization
restriction from the Knudsen number $\varepsilon$, and the space
discretization restriction from the oscillatory factor in the media
$\delta$. Our aim is to design a numerical scheme that behaves well in
both the highly oscillatory and the fluid regimes. More precisely, a
desirable algorithm would
\begin{itemize}
\item[1.] capture the diffusion limit with fixed discretization in the zero limit of the Knudsen number;
\item[2.] relax the discretization from the oscillation indicator $\delta$ while maintaining the macroscopic quantities.
\end{itemize}

We will follow the principles of AP and numerical homogenization, that
is to look for implicit solvers and apply upscaled basis
functions. However, a straightforward combination is not sufficient to
capture the limiting regime when both $\veps$ and $\delta$ are small.
If we simply use the basis obtained from numerical homogenization, in
the limit $\varepsilon\to 0$, the scheme does not converge to that of
the homogenized heat equation.

Considering the fact that the even and odd parts of the solution
play different roles in the diffusion limit, we treat them differently in the Galerkin projection to incorporate
the scattering coefficient $\sigma^{\delta}$. As will be shown later, such special and
dedicate treatment is the key that allows one to preserve the correct
discretization in the limiting heat equation regime, and it is the
main contribution of the current paper.

In the following, we describe our numerical method in Section 2, and
the convergence of the scheme is analyzed in Section 3. In Section 4
we conclude with some numerical examples.

\section{Numerical method}
We prescribe the algorithm in this section. To be more specific, we
discuss the numerical method for the problem in two spatial dimension
with particles traveling at the same speed so only the direction of
the velocity differs. Problem in other dimensions could be treated
similarly. We write the radiative transfer equation as
\begin{align}\label{eqn:radiative_transfer}
\varepsilon a^{\delta} \partial_t f +  a^{\delta} \cos{\xi}\partial_xf +  a^{\delta} \sin{\xi}\partial_yf = \frac{1}{\varepsilon}\mathcal{L}f\,,&\quad (\xx,\xi) = (x,y,\xi)\in\Omega\otimes(-\pi,\pi]\,.
\end{align}
where $a^{\delta} = \frac{1}{\sigma^{\delta}} = a(\frac{\xx}{\delta})$
is the inverse scattering coefficient, which is periodic with period
$\delta$. We use $a^{\delta}$ instead of $\sigma^{\delta}$ so that
the resulting diffusion limit takes the usual form of heat equation
with oscillatory coefficient, as will be shown below. We have assumed
periodic coefficient $a^{\delta} = a(\frac{\xx}{\delta})$; it is
straightforward to extend to two scale coefficients
$a^{\delta} = a(\xx, \frac{\xx}{\delta})$ where $a(\xx, \cdot)$ is
periodic. For simplicity we take the collision operator
\begin{equation}
\mathcal{L}f = \frac{1}{2\pi}\int f\rd{\xi} - f\,.
\end{equation}
The velocity domain is represented using the angle $\xi\in(-\pi,\pi]$:
$v = (\cos{\xi},\sin{\xi})$ gives the velocity. For arbitrary small
but fixed $\delta$, in the zero limit of $\varepsilon$ (recall that we
consider the regime $\veps \ll \delta \ll 1$ in this work), the
transport equation converges to the following heat equation with
highly oscillatory diffusion coefficient:
\begin{equation}\label{eqn:limit_heat}
\partial_t\rho^\delta = \frac{1}{2}\nabla_{\xx}\cdot\left( a^{\delta} \nabla_\xx\rho^\delta\right)\,,
\end{equation}
where $f(t,\xx,\xi)\to\rho^\delta(t,\xx)$ as $\varepsilon\to 0$,
having its velocity dependence vanishing in the diffusion limit. The
solution is still highly oscillatory in $\xx$ due to the heterogeneity
in space.

Sending $\delta\to 0$, we will obtain the homogenized limiting heat
equation. As $\delta\to 0$, $\rho^{\delta}\to\rho$ where $\rho$ solves
the homogenized heat equation with a smooth media:
\begin{equation}\label{eqn:limit_hom_heat}
\partial_t\rho = \frac{1}{2}\nabla_\xx\cdot\left( a_{\hom}\nabla_\xx\rho\right)\,.
\end{equation}
Here $a_{\hom}$ is the homogenized coefficient, which could be obtained by solving the cell problem \cite{BLP}
\begin{equation}
  \ee \cdot a_{\hom}   \ee = \inf_{\chi_{\ee}}  \int_{\Gamma} a(y) \bigl\lvert \nabla \chi_{\ee}(y) + \ee \bigr\rvert^2 \ud y, \qquad \forall \lvert \ee \rvert = 1, \ee \in \mathbb{R}^2.
\end{equation}
where $\Gamma$ is the unit cell of the periodic coefficient $a$. The
homogenized coefficient $a_{\hom}$ is a $2\times 2$ matrix and in
general is not isotropic. 

We seek for an algorithm that is accurate across regimes with the discretization independent on the external parameters such as $\varepsilon$ and $\delta$. More precisely, we look for a numerical scheme that captures accurate numerical solutions both in the kinetic regime with $\varepsilon = \mathcal{O}(1)$, and in the fluid regime with $\varepsilon \to 0$; with either smooth media where $\delta = \mathcal{O}(1)$ or highly oscillatory media with $\delta \to 0$.

Under the Galerkin framework, we construct some basis functions first
and then project the original equation~\eqref{eqn:radiative_transfer}
onto the finite dimensional space expanded by them. The convergence
will simply be governed by the effectiveness of the basis
functions. However, it turns out directly performing the projection is
not going to maintain the AP property, and a reformulation is
needed. Below we first describe the even-odd reformulation of the
equation, and the associated discretization. It will be followed by
the basis construction in subsections~\ref{sec:basis_xi}
and~\ref{sec:basis_x}.

\subsection{Reformulation via even-odd
  decomposition}\label{sec:evenodd}
The even-odd decomposition for the transport equation has been used
for obtaining AP property by many studies,
see~\cite{Klar,JPT2}. It turns out also useful in our context to
  capture simultaneously the diffusion and homogenization limits. Let
  us define the even and the odd part of the solution:
\begin{equation}
f^\text{E} = \frac{1}{2}\left[ f(t,x,\xi) + f(t,x,-\xi)\right]\,,\quad f^\text{O} = \frac{1}{2}\left[ f(t,x,\xi) - f(t,x,-\xi)\right]\,.
\end{equation}
It is obvious that $f = f^\text{E} + f^\text{O}$. With such decomposition we reformulate the equation~\eqref{eqn:radiative_transfer} as:
\begin{equation}\label{eqn:even_odd}
\begin{cases}
  \text{Even:}&\quad \displaystyle
   a^\delta\partial_t f^{\tE} + \frac{ a^\delta}{\varepsilon}v\cdot\nabla_x\fO = \frac{1}{\varepsilon^2}\left(\langle\fE\rangle_{\xi}-\fE\right); \\[1em]
  \displaystyle \text{Odd:}&\quad \displaystyle  a^\delta\partial_t \fO +
  \frac{ a^\delta}{\varepsilon}v\cdot\nabla_x\fE =
  -\frac{1}{\varepsilon^2}\fO.
\end{cases}
\end{equation}
Note in particular that the average in collision operator 
acting on the odd function gives $\langle \fO \rangle_{\xi} = 0$.

To ensure the asymptotic preserving property, implicit treatment has
to be applied on stiff terms, and here we will treat both the
convection and the reaction terms implicitly. Thus the resulting
  scheme is fully implicit. Taking backward Euler for example, given
the value of $f^{\tE,n}$ and $f^{\tO,n}$ at time step $t_n$, we solve
for the functions at the new time step by
\begin{equation}\label{eqn:even_odd_dis}
\begin{aligned}
\text{Even:}\quad && a^\delta f^{\tE,n+1} + \frac{\Delta t}{\varepsilon^2} f^{\tE,n+1} - \frac{\Delta t}{\varepsilon^2}\langle f^{\tE,n+1}\rangle_{\xi} &=  a^\delta f^{\tE,n}-\frac{\Delta t}{\varepsilon} a^\delta v\cdot\nabla_\xx f^{\tO,n+1}\,,\\
\text{Odd:}\quad && a^\delta f^{\tO,n+1} + \frac{\Delta t}{\varepsilon^2} f^{\tO,n+1} &=  a^\delta f^{\tO,n}-\frac{\Delta t}{\varepsilon} a^\delta v\cdot\nabla_\xx f^{\tE,n+1}\,.
\end{aligned}
\end{equation}
Here $\Delta t$ is the time step size.

To turn the above semi-discrete equation into a fully discretized
system, we now employ discretization in spatial and velocity domain.
Under the general Galerkin framework, we expand the solutions with
pre-constructed basis functions:
\begin{equation}\label{eqn:even_odd_notation}
f^\tE \sim f^E_{M,N} = \sum_{m,n=1}^{M,N} \alpha_{mn}\phi_m(\xx)p_n(\xi)\,,\quad f^\tO \sim f^O_{M,N} = \sum_{m,n=1}^{M,N} \beta_{mn}\phi_m(\xx)p_n(\xi)\,,
\end{equation}
Here we choose $M$ basis functions $\{\phi_m(\xx)\}$ in spatial domain and $N$ basis functions $\{p_n(\xi)\}$ in velocity space respectively. 

To update $\alpha_{mn}$ and $\beta_{mn}$, we substitute the ansatz
into~\eqref{eqn:even_odd_dis} and project the equation onto the finite
dimensional space
$\Span\{\phi_m(\xx)p_n(\xi), m = 1, \ldots, M, n = 1, \ldots, N\}$.
\textit{In fact the projection is not unique as we may change the form
  of the equations \eqref{eqn:even_odd_dis} before the projection.}
For consistency with the asymptotic limit, we divide the even equation
with $a^\delta$ before the Galerkin projection while keeping the form
of the odd equation in the projection. We emphasize that for the
  spatial and velocity discretization, the even and odd equations are
  treated differently. This is crucial for the scheme to capture both
  diffusion and homogenization limits, as will be shown below; see
  also Remark~\ref{rmk:asym_formulation}.

In a concise matrix form, we arrive at the discrete system 
\begin{equation}\label{eqn:Galerkin_E}
\left\{\Phi\otimes\II +\frac{\Delta t}{\varepsilon^2} \Sigma^\text{inv}\otimes(\II-\PP)\right\}\cdot\vec{\alpha}^{n+1}=\left(\Phi\otimes\II\right)\cdot\vec{\alpha}^n-\left\{\frac{\Delta t}{\varepsilon}\Xi^x\otimes\FF^{\cos}+\frac{\Delta t}{\varepsilon}\Xi^y\otimes\FF^{\sin}\right\}\cdot\vec{\beta}^{n+1}
\end{equation}
for the even function and
\begin{equation}\label{eqn:Galerkin_O}
\left\{\Sigma\otimes\II +\frac{\Delta t}{\varepsilon^2}\Phi\otimes\II\right\}\cdot\vec{\beta}^{n+1}=\left(\Sigma\otimes\II\right)\cdot\vec{\beta}^n-\left\{\frac{\Delta t}{\varepsilon}\Sigma^x\otimes\FF^{\cos}+\frac{\Delta t}{\varepsilon}\Sigma^y\otimes\FF^{\sin}\right\}\cdot\vec{\alpha}^{n+1}
\end{equation}
for the odd function. In this formulation we have used various mass
and stiffness matrices that given by:
\begin{align}\label{eqn:assemble_matrices}
\Sigma_{mn} &= \langle a^{\delta} \phi_m\,,\phi_n\rangle_\xx\,,\quad &\Phi_{mn} &= \langle\phi_m\,,\phi_n\rangle_\xx\,,\\
\Sigma^x_{mn} &= \langle\phi_m\,,a^{\delta} \partial_x\phi_n\rangle_\xx\,,\quad& \Sigma^y_{mn} &= \langle\phi_m\,,a^{\delta} \partial_y\phi_n\rangle_\xx\,,\nonumber\\
\Xi^x_{mn} &= \langle\phi_m\,,\partial_x\phi_n\rangle_\xx\,,\quad& \Xi^y_{mn} &= \langle\phi_m\,,\partial_y\phi_n\rangle_\xx\,,\nonumber\\
\Sigma^\text{inv}_{mn} &= \langle (a^\delta)^{-1}\phi_m\,,\phi_n\rangle_\xx\,,\nonumber
\end{align}
on the spatial domain and
\begin{align}\label{eqn:assemble_xi}
\II_{mn} & = \langle p_m\,,p_n\rangle_\xi\,,\quad & \BB_{mn} &= -\langle \mathcal{L}p_m\,,p_n\rangle_\xi\\
\FF^{\cos}_{mn} & = \langle \cos{\xi}\,p_m\,,p_n\rangle_\xi\,,\quad & \FF^{\sin}_{mn} &= \langle \sin{\xi}\,p_m\,,p_n\rangle_\xi\,,\nonumber\\
\PP_{mn} & = \langle p_m \rangle_{\xi} \langle p_n\rangle_\xi\,\nonumber
\end{align}
on the velocity domain. The coefficients $\alpha$ and $\beta$ needed to be ordered in a consistent way:
\begin{equation*}
\vec{\alpha} = [\alpha_{11}\,,\alpha_{12}\,,\cdots\,,\alpha_{1N}\,,\alpha_{21}\,,\cdots\,,\alpha_{2N}\,,\cdots,\alpha_{MN}]^T\,,\quad \vec{\beta} = [\beta_{11}\,,\beta_{12}\,,\cdots\,,\beta_{1N}\,,\cdots,\beta_{MN}]^T\,.
\end{equation*}

The basis functions along the velocity domain determine the structure
of $\II$, $\BB$, $\PP$ and the two flux terms $\FF^{\cos}$ and
$\FF^{\sin}$, while $\Sigma$, $\Phi$, $\Sigma^\text{inv}$ and the
four flux terms $\Sigma^x$, $\Sigma^y$, $\Xi^x$ and $\Xi^y$ are
determined by the basis construction along the spatial domain. To avoid confusion we use Greek letters for $\xx$ and Latin letters for $\xi$. Note that the flux matrices are in general not symmetric or anti-symmetric, due to the presence of $a^{\delta}$.

\begin{remark}\label{rmk:asym_formulation}
If we keep the form of the even equation in the projection, we will get  the following updating formula instead (cf. \eqref{eqn:Galerkin_E}):
\begin{equation}\label{eqn:Galerkin_asym_E}
\left\{\Sigma\otimes\II +\frac{\Delta t}{\varepsilon^2}\Phi\otimes(\II-\PP)\right\}\cdot\vec{\alpha}^{n+1}=\left(\Sigma\otimes\II\right)\cdot\vec{\alpha}^n-\left\{\frac{\Delta t}{\varepsilon}\Sigma^x\otimes\FF^{\cos}+\frac{\Delta t}{\varepsilon}\Sigma^y\otimes\FF^{\sin}\right\}\cdot\vec{\beta}^{n+1}\,.
\end{equation}
In terms of computation this formulation might be easier
than~\eqref{eqn:Galerkin_E} since we save the computation of three
more terms: $\Xi^x$, $\Xi^y$ and $\Sigma^\text{inv}$. However, as will
be seen in Section~\ref{sec:convergence}, such discretization fails to
capture the asymptotic limit, and also leads to an asymmetric discretization of the limiting heat equation. 
\end{remark}

In the following two subsections we briefly describe the basis function construction along $\xx$ and $\xi$ respectively, and the associated numerical integration called for in evaluating the coefficients in~\eqref{eqn:assemble_matrices} and~\eqref{eqn:assemble_xi}. These basis need to be constructed such that: (1) the corresponding matrices enjoy simple structure, and (2) the high oscillation in the scattering coefficient is captured.

\subsection{Basis functions in $\xi$}\label{sec:basis_xi}

To construct basis functions along the velocity space, we use the standard Pn method. This is a well-accepted method for kinetic type of equations, especially for radiative transfer equations.

In short, Pn method uses the Legendre polynomials as basis
functions. They are a set of orthogonal polynomials in a bounded
domain with uniform weight functions:
\begin{equation}
\frac{1}{2\pi}\int_{-\pi}^{\pi} p_n(\xi)p_m(\xi)\rd{\xi} = \delta_{mn}\,.
\end{equation}
Here $p_n$ is a normalized $(n-1)$-th order polynomial in $\xi$, and they
are orthogonal with respect to each other. Some advantages of the
method are immediate. The set simultaneously diagonalizes two
operators in the equation: both $\II$ and $\BB$ are diagonal
matrices. $\II$ being diagonal is easy to see due to the definition,
and $\BB$ is diagonal mainly due to the structure of the collision
operator. The Legendre polynomials are the eigenfunctions of
$\mathcal{L}$:
\begin{equation}
\mathcal{L}p_n(\xi)  = \lambda_np_n(\xi)\,,
\end{equation}
and for the collision term in equation~\eqref{eqn:radiative_transfer} specifically, we have:
\begin{equation}
\mathcal{L}p_n(\xi)  = \begin{cases}
0\,,\quad &n=1\,,\\
-p_n(\xi)\,,\quad &n\neq 1\,.
\end{cases}
\end{equation}
And thus:
\begin{equation}\label{eqn:BBPP}
\BB = \left(\begin{array}{cccc}0 & \cdots & \cdots & \cdots\\
0 & 1 & 0 & \cdots\\
\vdots & \ddots & \ddots & \vdots\\
0 & \cdots & 0 & 1
\end{array}\right)\,,\quad\PP = \left(\begin{array}{cccc}1 & 0 & \cdots & \cdots\\
0 & 0 & 0 & \cdots\\
\vdots & \ddots & \ddots & \vdots\\
0 & \cdots & 0 & 0
\end{array}\right)\,,
\end{equation}
meaning $\BB$ is an identify matrix except the $(1,1)$-entry is changed zero, and $\PP$ is a zero matrix except the $(1,1)$-entry is $1$. This prior knowledge saves us from performing numerical integration for assembling stiffness matrices. The flux terms, however, requires numerical integration.

In 1D, the form of the flux term could be further simplified.  It
reads as
\begin{equation}
\FF = \langle \xi p_n(\xi)\,,p_m(\xi)\rangle_\xi\,.
\end{equation}
According to the definition of the Legendre polynomial, the set satisfies the recurrence relation, and the flux matrix $\FF$ is a tridiagonal matrix.

In higher dimensions the flux terms no longer have such good structure and to precompute the flux terms $\FF^{\cos}$ and $\FF^{\sin}$, one needs to perform numerical integration. Here we utilize another hidden benefit of using orthogonal polynomial: the numerical integral is highly accurate with the Gaussian quadratures. Suppose we sample $K$ grid points on the velocity domain, the Gaussian quadratures are then the zeros for the $(K-1)$-th Legendre polynomials. We denote the sample points and the associated weights $\{\xi_k\,,w_k\}$ with $k=1,\cdots K$, then the integrations are computed as:
\begin{equation}\label{eqn:fluxes}
\FF^{\cos}_{mn} = \langle \cos{\xi}\, p_m(\xi)\,,p_n(\xi)\rangle_\xi \sim \sum_{k=1}^K\cos{\xi_k} p_m(\xi_k)p_n(\xi_k)w_k\,.
\end{equation}
The same computation holds true for $\FF^{\sin}$. This finishes our preparation on the velocity domain.

\subsection{Basis functions in $\xx$}\label{sec:basis_x}

For constructing basis functions along the spatial domain, we borrow
ideas from numerical homogenization to characterize the highly oscillatory media.

The numerical homogenization and upscaling has been studied thoroughly
for elliptic / parabolic type equations with highly oscillatory
heterogeneous
media. Among many techniques in numerical homogenization, we choose to use
the multiscale finite element method (MsFEM) to construct basis
functions. The possible adaptation of other techniques will be left to
future research.

The idea of MsFEM is to decompose the domain into nested grids, with the basis functions constructed on coarse mesh using fine grids. The basis functions, by construction, expand the null space of the elliptic operator patchwisely. As shown in~\eqref{eqn:limit_heat}, in the limiting regime, the elliptic operator for the diffusion equation is $\nabla_\xx\cdot(a_{\hom}\nabla_\xx\cdot)$. Following MsFEM, we construct a nested coarse-fine grids system with $D_H \in D_h$. Here $D_H = \{\xx_1\,,\cdots, \xx_N\}$ is the collection of coarse grid points with mesh size $H$ and $D_h$ collects all fine grid points with mesh size $h$. Typically $H$ is assumed not to resolve the fine scale $\delta$ but $h$, the fine mesh needs to. The subdomains are referred to the triangulations formed by the coarse grid points $D_H$, for example we could use $\mathcal{T}_H = \{D_H^k , 1 \leq k \leq K \}$ to denote the finite set of $D^k_H$, compact triangles or quadrilaterals constructed using coarse grid points. Here we assume there are $K$ subdomains in total, and the union of these $K$ subdomains covers the closure of the entire domain $D$. The intersection of different triangles or quadrilaterals is either empty, a common node or a common edge. Similarly we denote $\mathcal{T}_h$ the collections of the triangulation on the fine scale $\mathcal{T}_h = \{D_h^k , 1 \leq k \leq K_h \}$.

The essential idea is to precompute the Multi-scale Finite Element
Basis (MsFEB) in these subdomains using fine grid points, and assemble
the stiffness matrix with them. The Galerkin formulation is performed
on these basis functions that are associated with coarse mesh. Suppose
patch $D^m_H$ has $d$ nodal grid points, denoted as
$\xx_1, \ldots, \xx_d$, then in this patch, we construct $d$ basis
functions, with each one being associated with one nodal grid point:
\begin{equation*}
\begin{cases}
-\nabla_\xx\cdot(a^\delta \nabla_\xx\phi^l_m) = 0\,,\quad &\xx\in D^m_H\,,\\
\phi_m^{l}(\xx) = \psi_m^l(\xx)\,,& \xx\in\partial D^m_H\,,\quad l = 1,\cdots d\,.
\end{cases}
\end{equation*}
Here the boundary condition $\psi_m^l$ is set such that it sets $1$ at  grid point $\xx_l$ and $0$ for all others from the same patch:
\begin{equation}\label{eqn:boundary_MsFEM}
\psi_m^l(\xx_j) = \delta_{lj}\,, 
\end{equation}
and $\psi^l_m(\xx_j)$ is affine on the boundary, i.e., they are
  taken to be hat functions, restricted to the patch $D^m_H$.  As
seen in the formulation, the equation is computed in subdomain
$D^m_H$, with boundary conditions that the basis function ``picks up''
one nodal point of the patch. Obviously the basis functions $\phi^l_m$
computed here resembles the hat functions (restricted in a
  single patch) in the standard FEM that also takes value $1$ or $0$
at nodal grid points, but these basis functions have more details
embedded and thus the coarse mesh size $H$ does not need to resolve
the oscillation parameter $\delta$.

\begin{remark}
In the framework of MsFEM, other choices of the functions
    $\psi_m^l$ are possible: After we fix the nodal values as
    \eqref{eqn:boundary_MsFEM}, several possibilities exists for the
    choice of boundary values on the edge of the patch; such as the
    linear boundary conditions (which is what we used above)and the
    oscillatory boundary conditions (that is to compute the elliptic
    equation confined on the edges for the boundary values). Another
    choice recommended in~\cite{HW97} is to use the over-sampling
    technique: we first compute basis functions on an enlarged patch
    with linear boundary conditions, and then restrict the solutions
    obtained in the original smaller subdomain $D^k_H$. We have chosen
    the linear boundary conditions here for simplicity, while other
    choices are possible.
\end{remark}

Each nodal grid point, if not on the boundary, appear in multiple patches, and after the computation of all basis functions in all patches, for each nodal grid point $\xx_l$, we sum them up and obtain:
\begin{equation}\label{eq:defphil}
  \phi_l = \sum_{m: \, \xx_l \in D_H^m}\phi^l_m\,.
\end{equation}
They serve as the basis functions in the Galerkin formulation. Detailed construction could be found in the original paper~\cite{HW97}. With these basis functions constructed, we are ready to assemble the stiffness matrices in~\eqref{eqn:Galerkin_E} and~\eqref{eqn:Galerkin_O}. Here we need to compute $\Sigma$, $\Sigma^x$, $\Sigma^y$, $\Sigma^\text{inv}$, $\Xi^x$ and $\Xi^y$. To obtain the numerical integration, considering the basis functions are defined on fine grid points in local patches, we could simply use the very basic trapezoidal rule, for example:
\begin{equation}\label{eqn:Sigma}
\Sigma_{mn} = \langle a^\delta \phi_m\,,\phi_n\rangle_\xx\sim\sum_{k\in\mathcal{T}_h}\big|D^k_h\big|\left(a^\delta \phi_m\phi_n\right)\big|_{D^k_h}
\end{equation}

Once the basis functions are constructed and the stiffness matrices are assembled, there are no further use of the fine grid points and we could neglect them. This finishes our preparation in the spatial domain.

With the stiffness matrices computed in~\eqref{eqn:fluxes} and~\eqref{eqn:Sigma} we evolve equation~\eqref{eqn:Galerkin_E} and~\eqref{eqn:Galerkin_O} for the projection coefficients $\vec{\alpha}^{n+1}$ and $\vec{\beta}^{n+1}$.

\section{Convergence}\label{sec:convergence}
The success of the method lies in the two main ingredients. The
different treatments of the even and the odd equation shown
in~\eqref{eqn:Galerkin_E} and~\eqref{eqn:Galerkin_O}, and the
construction of the basis functions discussed in
subsection~\ref{sec:basis_x}. In this section we discuss the
properties of the scheme, mainly to show that it is asymptotic
preserving and captures the homogenized limit.  These two properties
combined ensures the convergence of the method with the discretization
$H$ and $\Delta t$ relaxed from both the two small scales
$\delta$ and $\varepsilon$ respectively.

To justify this numerical method we need to show the error
\begin{equation}\label{eqn:error}
\text{Error} = f - f^E_{M,N} - f^O_{M,N}
\end{equation}
is small where $f^E_{M,N}$ and $f^O_{M,N}$, defined
in~\eqref{eqn:even_odd_notation}, are determined by the coefficients
$\vec{\alpha}$ and $\vec{\beta}$ through
updating~\eqref{eqn:Galerkin_E} and~\eqref{eqn:Galerkin_O}.  Since the
result is trivial for $\varepsilon\sim\delta =\mathcal{O}(1)$, we only
focus on the case where parameters are small. As mentioned, we assume
the regime that $1\gg\delta \gg \varepsilon$.

Generally speaking it is not easy to control the error
in~\eqref{eqn:error}, especially given the undetermined role of the
two parameters in~\eqref{eqn:Galerkin_E} -\eqref{eqn:Galerkin_O}. As
seen before, the basis functions $\phi_m$ are constructed in a special
way such that the oscillation in the media gets embedded in: they are
constructed as $a$-harmonic function in each element. To estimate the
error, we will resort to the diffusion limit for which the basis
functions work well as in standard MsFEM method.

To this end, we first write the error term as:
\begin{equation}
\|f - f^E_{M,N} - f^O_{M,N}\| \leq \underbrace{\|f-\rho^{\delta}\|}_{\text{term I}} + \underbrace{\|\rho^{\delta} - \rho\|}_{\text{term II}}+ \underbrace{\|\rho-f^E_{M,N}-f^O_{M,N}\|}_{\text{term III}}\,. 
\end{equation}
where $\rho^{\delta}$ solves the diffusion limit~\eqref{eqn:limit_heat} and $\rho$ is the solution to the homogenized heat equation limit~\eqref{eqn:limit_hom_heat}. We summarize the three terms below and lay out the strategy for the proof. Without further notice, the norm of the error terms will always be choose as $L_2$ norm, either $L_2(\rd{\xx})$ or $L_2(\rd{\xx}\rd{v})$ depending on the contexts.
\begin{itemize}
\item Term I: it is the comparison between the solution to the
  transport equation with the diffusion limit.  For fixed $\delta$ and
  small $\varepsilon$, it is expected to be as small as
  $\mathcal{O}(\varepsilon)$ in the asymptotic limit. We cite in the next 
  theorem the classical result from~\cite{BSS84}.
\end{itemize}

\begin{theorem}\label{thm:diff_limit}
  Let $f$ and $\rho^{\delta}$ solve
  equation~\eqref{eqn:radiative_transfer} and~\eqref{eqn:limit_heat}
  respectively, with periodic boundary conditions. Assume the
  initial data for $f$ has no dependence on $v$, we have
  \begin{align*}
    \|f - \rho^{\delta}\|_{L_2(\rd{x}\rd{v})} = \mathcal{O}(\varepsilon)\,,
  \end{align*}
\end{theorem}

\begin{remark}
  The periodic boundary condition excludes the complexity of the
  boundary layer effect, which we will not address in this work. The
  requirement of initial data independent on $v$ also exclude the
  initial layer. When initial layer exists, due to the exponential
  decay, it induces an error of order
  $\mathcal{O}(e^{-t/\varepsilon^2})$.
\end{remark}

\begin{itemize}
\item Term II: it represents the homogenization error. With $\delta\to 0$, the standard homogenization theory of heat equation~\cite{BLP} indicates that the error here is of order $\mathcal{O}(\delta)$. 
\item Term III: this is the error coming from numerical
  discretization. Typical brute-force analysis would pessimistically
  give error bounds depending on $\frac{\Delta x}{\delta}$ or
  $\frac{\Delta t}{\varepsilon}$. As will be shown later this is not
  the case due to the special design of the scheme: we demonstrate
  that the method captures the numerical homogenization limit with
  fixed discretization, besides being asymptotic preserving.
    Theorem~\ref{thm:diff_limit_discrete} guarantees that this error
    can be bounded by $\mathcal{O}(\varepsilon + \sqrt{\delta} + \Delta t + H^2)$.
\end{itemize}

We summarize the result here first by collecting the estimates
  for the three error terms.

\begin{theorem}
  Let $f$ be the solution to the
  equation~\eqref{eqn:radiative_transfer} with initial data
  independent of $v$ and in $C^3$. Let $f^E_{M,N} + f^O_{M,N}$ be the numerical approximation computed
  through~\eqref{eqn:Galerkin_E} and~\eqref{eqn:Galerkin_O}. Given
  oscillatory but periodic scattering coefficient
  $\sigma(\xx) =\sigma(\xx/\delta)$, we have: 
  \begin{equation}
    \|f-f^E_{M,N} - f^O_{M,N}\|_{L_2(\rd{x}\rd{v})}\leq \mathcal{O}(\varepsilon +  \sqrt{\delta} + \Delta t + H^2)\,.
  \end{equation}
\end{theorem}

This result could be improved in 1D, as seen in Remark~\ref{rmk:conv_1D}. For later convenience, we first study the discretization of $\rho$ and $\rho^\delta$. Following the philosophy of asymptotic preserving scheme, we characterize the limiting numerical scheme as $\veps \to 0$. The error analysis of the third term will follow from the approximation error of the limiting scheme to the homogenized heat equation.

\subsection{Multiscale finite element method for the heat equation}

Let us take a detour and recall the Galerkin approximation for the
heat equation using the multiscale finite element basis constructed
before. While our scheme does not converge to a standard MsFEM scheme,
it will be useful to compare with it. In MsFEM, we approximate the
solution to the heat equation as
\begin{equation}
\rho^{\delta} \sim\rho^\delta_{M} = \sum_{m=1}^M\eta_{m}\phi_m\,.
\end{equation}
In the Galerkin framework, we project \eqref{eqn:limit_heat} onto the finite dimensional space spanned by $\{\phi_m\}$, and the numerical scheme reads:
\begin{equation}\label{eqn:Galerkin_rho}
\Phi\cdot\partial_t\vec{\eta} - \frac{1}{2}\AA\cdot\vec{\eta} = 0\,,
\end{equation}
where $\vec{\eta} = [\eta_1\,,\cdots,\eta_M]^T$ and 
\begin{align}
\Phi_{nm} &= \langle\phi_m\,,\phi_n\rangle_\xx\,,\quad \nonumber\\
\AA_{nm} &= \langle \nabla_\xx\cdot( a^\delta\nabla_\xx\phi_m)\,,\phi_n\rangle_\xx = -\langle a^\delta\nabla_\xx\phi_m\,,\nabla_\xx\phi_n\rangle_\xx\,.\label{eqn:AA}
\end{align}
Note that the definition of $\Phi$ is the same as the mass matrix
defined in~\eqref{eqn:assemble_matrices}. The stiffness matrix $\AA$,
by definition, is symmetric. The equation~\eqref{eqn:Galerkin_rho}, as
a semi-discretization of the heat equation, provides the evolution of
$\vec{\eta}$, the projection coefficients for $\rho^\delta$.

For a full discretization, we suppose at time step $t_n$ we have $\rho^{\delta,  n}_M$ ready. The simplest method for updating for the new time $\rho^{\delta,n+1}_M$ equation~\eqref{eqn:Galerkin_rho} that avoids parabolic time step size restriction is the backward Euler method:
\begin{equation}\label{eqn:rho_dis}
\Phi\vec{\eta}^{n+1} = \Phi\vec{\eta}^{n} +\frac{\Delta t}{2}\AA\vec{\eta}^{n+1}\,.
\end{equation}
For updating~\eqref{eqn:rho_dis}, one needs to find a numerical solver that efficiently and accurately invert $\Phi - \frac{\Delta t}{2}\AA$.

We also discretize the homogenized heat limit~\eqref{eqn:limit_hom_heat}. Following the standard Galerkin formulation, we use the simplest finite elements, namely when it is confined in $l$-th patch, it satisfies:
\begin{equation*}
\begin{cases}
-\nabla_\xx\cdot(a_{\hom}\nabla_\xx\bar{\phi}^l_m) = 0\,,\quad &\xx\in D^m_H\,,\\
\bar{\phi}_m^{l}(\xx) = \psi_m^l(\xx)\,,& \xx\in\partial D^m_H\,,\quad l = 1,\cdots d\,.
\end{cases}
\end{equation*}
In 1D, they are simply the hat function. For simplicity of the
analysis, we assume that $\psi_m^l$ is linear on the boundary for
higher dimensional cases, so that $\bar{\phi}_l$ constructed similarly
as \eqref{eq:defphil} are also (higher dimensional) hat functions.
The following lemma is from the standard MsFEM analysis
\begin{lemma}\label{lemma:AAPhi}
Define the homogenized stiffness and mass matrices:
\begin{equation}
  (\AA_{\hom})_{nm} = -\langle a_{\hom}\nabla_\xx\bar{\phi}_m\,,\nabla_\xx\bar{\phi}_n\rangle_\xx\,,\quad (\Phi_{\hom})_{nm} = \langle \bar{\phi}_m\,,\bar{\phi}_n\rangle_\xx\,,
\end{equation}
we then have 
\begin{equation*}
\lvert (\AA_{\hom})_{nm} - \AA_{nm} \rvert = \mathcal{O}(\sqrt{\delta})\,,\quad \lvert (\Phi_{\hom})_{nm} - \Phi_{nm} \rvert = \mathcal{O}(\delta),
\end{equation*}
where $\AA$ and $\Phi$ are defined in \eqref{eqn:AA}.
\end{lemma}

\begin{proof}
  We first recall from standard periodic homogenization (e.g.,
  \cite{BLP} or in the context of MsFEM \cite{HWC99}) of elliptic
  equations
  \begin{equation}\label{eq:convrate}
    \lVert \phi_m(\xx) -
    \bar{\phi}_m(\xx) - \delta \chi_x(\xx/\delta)
    \partial_x \bar{\phi}_m(\xx) - \delta \chi_y(\xx/\delta) 
    \partial_y \bar{\phi}_m(\xx) \rVert_{H^1(\rd x)} = \begin{cases}
      \mathcal{O}(\delta)\,,\quad &\text{1D}\\
      \mathcal{O}(\sqrt{\delta})\,,\quad & \text{higher D}
    \end{cases}\,,
  \end{equation}
  where $\chi_x$ and $\chi_y$ are correctors for the periodic
  homogenization. The lower rate of convergence in higher dimension is
  caused by boundary layers.  The limits
  $\AA_{\hom} = \lim_{\delta\to 0}\AA$ and
  $\Phi_{\hom} = \lim_{\delta\to 0}\Phi$ thus follow since
  $\phi_m\to\bar{\phi}_m$,
  $\nabla \phi_m \to  \nabla \bar{\phi}_m$ and
  $a \nabla \phi_m \rightharpoonup a_{\hom}\nabla\bar{\phi}_m$ as
  $\delta \to 0$. The convergence rate also follows from \eqref{eq:convrate} easily. 
\end{proof}

This naturally leads to the consistency of MsFEM to the
  homogenized heat equation. The proof is straightforward based on the
  previous Lemma, which we omit here.
\begin{proposition}\label{lemma:heat}
  As $\delta \to 0$, the multiscale finite element method for the
  equation~\eqref{eqn:limit_heat} converges to the following scheme:
  \begin{equation}\label{eq:scheme_hom_heat}
    \Phi_{\hom}\vec{\eta}^{n+1} = \Phi_{\hom}\vec{\eta}^{n} +\frac{\Delta t}{2}\AA_{\hom}\vec{\eta}^{n+1}\,.
  \end{equation}
  The limiting scheme is a consistent and stable discretization of the
  homogenized heat equation~\eqref{eqn:limit_hom_heat}.
\end{proposition}

\subsection{Term III: numerical homogenization and AP}
This subsection is devoted to showing the AP property, namely we would like to control the error between the transport equation numerical solution and the heat equation numerical solution, and the error should be independent of either $\varepsilon$ or $\delta$.

Before we turn to the limiting numerical scheme in the diffusion
  limit, let us characterize the limit of the coefficients in the
  following Lemma, which is analogous to Lemma~\ref{lemma:AAPhi} for
  MsFEM.

\begin{lemma}\label{lemma:limit_mat}
Define $\Xi_{\hom}^{x}$, $\Xi_{\hom}^y$, $\Sigma_{\hom}^x$ and $\Sigma_{\hom}^y$ the same way as defined in~\eqref{eqn:assemble_matrices} with $\phi_m$ replaced by $\bar{\phi}_m$, then
\begin{align}
\Xi_{\hom}^x = \lim_{\delta\to 0}\Xi^x\,,\quad\Xi_{\hom}^y = \lim_{\delta\to 0}\Xi^y\,,\\
\Sigma_{\hom}^x = \lim_{\delta\to 0}\Sigma^x\,,\quad\Sigma_{\hom}^y = \lim_{\delta\to 0}\Sigma^y\,.
\end{align}
Furthermore, let 
\begin{equation}\label{def:DD}
\DD = (\Xi^x\cdot\Phi^{-1}\cdot\Sigma^x)+(\Xi^y\cdot\Phi^{-1}\cdot\Sigma^y)\,, 
\end{equation}
then the limit $\DD_{\hom} = \lim_{\delta \to 0}\DD$ exists and is given by 
\begin{equation}\label{eqn:DD_hom}
  \DD_{\hom} = (\Xi_{\hom}^x\cdot\Phi_{\hom}^{-1}\cdot\Sigma_{\hom}^x)+(\Xi_{\hom}^y\cdot\Phi_{\hom}^{-1}\cdot\Sigma_{\hom}^y)\,.
\end{equation}
As $\delta \to 0$, we have 
\begin{equation}
  \lvert (\DD_{\hom})_{nm} - \DD_{nm} \rvert = \mathcal{O}(\sqrt{\delta}).
\end{equation}
\end{lemma}
\begin{proof}
The proof is similar to Lemma~\ref{lemma:AAPhi}: By using \eqref{eq:convrate}
\begin{align*}
\Sigma^x_{mn} \to \langle \bar{\phi}_m, (a_{\hom} \nabla \bar{\phi}_n)_x \rangle_\xx = (\Sigma_{\hom}^x)_{mn}\,; \quad \Xi^x_{mn} \to \langle \partial_x \bar{\phi}_m, \bar{\phi}_n \rangle_\xx = (\Xi_{\hom}^x)_{mn}\,, 
\end{align*}
with convergence rate $\Or(\sqrt{\delta})$ and similarly for $y$
direction. Here $(a_{\hom} \nabla \bar{\phi}_n)_x$ denotes the
$x$-component of the vector field $a_{\hom} \nabla \bar{\phi}_n$,
which is
$a_{\hom}^{xx}\partial_x\bar{\phi}_n+a_{\hom}^{xy}\partial_y\bar{\phi}_n$. The
conclusion for $\DD$ then follows immediately.
\end{proof}

We now ready to state the main result of this section, which
  concerns the limiting scheme of \eqref{eqn:Galerkin_E}
  and~\eqref{eqn:Galerkin_O} as $\veps$ and $\delta$ go to $0$. The
  result is analogous to Proposition~\ref{lemma:heat}.
\begin{theorem}\label{thm:diff_limit_discrete}
  Consider the scheme~\eqref{eqn:Galerkin_E}
  and~\eqref{eqn:Galerkin_O} for $f^E_{M,N}$ and $f^O_{M,N}$ with the
  multiscale finite element basis, as $\varepsilon\to 0$,
  $\delta\to 0$, the scheme converges to  a
  consistent and stable numerical methods for
  the homogenized heat equation~\eqref{eqn:limit_hom_heat}. More specifically,
\begin{itemize}
\item[(1)] $\vec{\beta}^{\varepsilon,\delta}\to 0$, as $\varepsilon\to 0$;
\item[(2)] $\alpha_{m,n}^{\varepsilon,\delta}\to 0$ for all $m$ with $n>1$ as $\varepsilon\to 0$;
\item[(3)] In the zero limit of $\varepsilon$, $\alpha^{\varepsilon,\delta}_{\cdot,1}\to\alpha^{\delta}_{\cdot,1}$ that satisfies:
\begin{equation}\label{eqn:alpha_1_delta}
\Phi\cdot\vec{\alpha}^{\delta,n+1}_1 = \Phi\cdot\vec{\alpha}^{\delta,n}_1 + \frac{\Delta t}{2}\DD\cdot\vec{\alpha}^{\delta,n+1}_1\,,
\end{equation}
where
$\vec{\alpha}^\delta_1 =
[\alpha^\delta_{1,1}\,,\alpha^\delta_{1,1}\,,\cdots\,,\alpha^\delta_{M,1}]$.
\item[(4)] The convergence of $\vec{\beta}$ and $\alpha$ is of
  $\mathcal{O}(\varepsilon)$, meaning
  $\vec{\beta}^{\varepsilon,\delta} = \mathcal{O}(\varepsilon)$,
  $\alpha_{m,n}^{\varepsilon,\delta} = \mathcal{O}(\varepsilon)$ for
  all $m$ and $n>1$, and
  $\vec{\alpha}^{\varepsilon,\delta}_{\cdot,1} -
  \vec{\alpha}^{\delta}_{\cdot,1} = \mathcal{O}(\varepsilon)$.
\item[(5)] In the zero limit of $\delta$, the scheme for $\vec{\alpha}^\delta_1$ converges to that of $\vec{\alpha}_1$ that satisfies:
\begin{equation}\label{eqn:limit_hom_dis}
\Phi_{\hom}\cdot\vec{\alpha}^{n+1}_1 = \Phi_{\hom}\cdot\vec{\alpha}^n_1 + \frac{\Delta t}{2}\DD_{\hom}\cdot\vec{\alpha}^{n+1}_1 \,.
\end{equation}
\item[(6)] The convergence is of $\mathcal{O}(\sqrt{\delta})$, meaning: $\vec{\alpha}_1^\delta - \vec{\alpha}_1 = \mathcal{O}(\sqrt{\delta})$.
\item[(7)] The scheme~\eqref{eqn:limit_hom_dis} is a consistent and
  stable scheme for the homogenized heat
  equation~\eqref{eqn:limit_hom_heat} with the convergence rate being
  $\mathcal{O}(\Delta t + H^2)$.
\end{itemize}
In summary, $\alpha^{\varepsilon,\delta}$ discretizes the
  limiting equation~\eqref{eqn:limit_hom_heat} with error
  $\mathcal{O}(\varepsilon + \sqrt{\delta} + \Delta t + H^2)$, where
  the first two terms are approximation error and the last two are
  discretization error.
\end{theorem}


The scheme~\eqref{eqn:limit_hom_dis} is not the same as the
scheme~\eqref{eq:scheme_hom_heat} for the homogenized heat equation,
which is the homogenized limit of the MsFEM
scheme~\eqref{eqn:Galerkin_rho}.  In general, the matrices
$\AA_{\hom}$ and $\DD_{\hom}$ are not the same.  In fact, the
scheme~\eqref{eqn:limit_hom_dis} is not even a Galerkin scheme. To
analyze the scheme, we will treat it more like a finite difference
approximation to the homogenized heat equation. The consistency is
given by the following lemma.

\begin{lemma}\label{lemma:1st_derivative} 
Let $f$ be a function in $C^3$. Let
$$f_n = f(\xx_n)\,,\quad g_n = a_{\hom}^{xx}\partial_x f(\xx_n)+a_{\hom}^{xy}\partial_y f(\xx_n).$$
Denote $\vec{f} = [f_1\,,\cdots\,,f_n]$ and $\vec{g} = [g_1\,,\cdots\,,g_n]$. then
\begin{equation}
\Phi^{-1}_{\hom}\cdot\Sigma^x_{\hom}\cdot \vec{f}-\vec{g} = \mathcal{O}(H^2)\,,
\end{equation}
where $H$ is the coarse mesh size of the discretization. Similarly,
$\Phi^{-1}_{\hom}\cdot\Sigma^y_{\hom}$ is an $\mathcal{O}(H^2)$
approximation to $a_{\hom}^{yx}\partial_x +a_{\hom}^{yy}\partial_y$.
\end{lemma}

\begin{proof}
Note that $\Phi_{\hom}$ is invertible with bounded inverse of $\mathcal{O}(1)$, it thus suffices to show that \begin{equation}
\Sigma^x_{\hom}\cdot \vec{f}-\Phi_{\hom}\cdot\vec{g} = \mathcal{O}(H^2)\,,
\end{equation}
meaning for each entry we need to show:
\begin{equation}
\left(\Sigma^x_{\hom}\cdot \vec{f}\right)_m = \sum_n f_n\langle\bar{\phi}_m\,,a^{xx}_{\hom}\partial_x\bar{\phi}_n+a^{xy}_{\hom}\partial_y\bar{\phi}_n\rangle_\xx
\end{equation}
is close enough to:
\begin{equation}
\left(\Phi_{\hom}\cdot\vec{g}\right)_m = \sum_n g_n\langle\bar{\phi}_m\,,\bar{\phi}_n\rangle_\xx\,.
\end{equation}
Since $\bar{\phi}$ is piecewise bilinear function in 2D, for $f$ in
$C^3$, standard interpolation results yield that
\begin{equation}
f - \sum_n f_n\bar{\phi}_n = \mathcal{O}(H^2)\,,\quad a^{xx}_{\hom}\partial_xf+a^{xy}_{\hom}\partial_yf - \sum_n g_n\bar{\phi}_n = \mathcal{O}(H^2)\,.
\end{equation}
As a result, we have
\begin{align*}
\sum_n f_n\langle\bar{\phi}_m\,,a^{xx}_{\hom}\partial_x\bar{\phi}_n+a^{xy}_{\hom}\partial_y\bar{\phi}_n\rangle_\xx &= -\sum_nf_n\langle a^{xx}_{\hom}\partial_x\bar{\phi}_m+a^{xy}_{\hom}\partial_y\bar{\phi}_m\,,\bar{\phi}_n \rangle_\xx\\
&= -\langle a^{xx}_{\hom}\partial_x\bar{\phi}_m+a^{xy}_{\hom}\partial_y\bar{\phi}_m\,,f \rangle_\xx + \mathcal{O}(H^2)\\
& = \langle \bar{\phi}_m\,,a^{xx}_{\hom}\partial_x f+a^{xy}_{\hom}\partial_yf \rangle_\xx + \mathcal{O}(H^2)\\
& =\sum_ng_n\langle \bar{\phi}_m\,,\bar{\phi}_n \rangle_\xx + \mathcal{O}(H^2)\,,
\end{align*}
which completes the proof.
\end{proof}
\begin{remark}\label{rmk:derivatives}
Following the same proof, we could see that $\mathcal{O}(H^2)$ error is produced if we use the following discretization for the approximation to the corresponding differential operator:
\begin{align}
& \partial_x\quad\rightarrow\quad -\Phi_{\hom}^{-1}\cdot\Xi^x_{\hom}\\
& \partial_y\quad\rightarrow\quad -\Phi_{\hom}^{-1}\cdot\Xi^y_{\hom}\,.
\end{align}
Note that the derivatives are on the second argument in the inner product thereby gives  a negative sign. Combined with the previous lemma, this means for $f\in C^2$,
\begin{align}
\Phi^{-1}_{\hom}\cdot\DD_{\hom} = \Phi^{-1}_{\hom}\cdot\Xi^x_{\hom}\cdot\Phi^{-1}_{\hom}\cdot\Sigma^x_{\hom} + \Phi^{-1}_{\hom}\cdot\Xi^y_{\hom}\cdot\Phi^{-1}_{\hom}\cdot\Sigma^y_{\hom} 
\end{align}
approximates $\nabla\cdot(a_{\hom}\nabla)$ with $\mathcal{O}(H^2)$ accuracy.
\end{remark}

With this lemma we could show the proof for Theorem~\ref{thm:diff_limit_discrete}.
\begin{proof}[Proof of Theorem~\ref{thm:diff_limit_discrete}]
We first perform asymptotic expansion of the two equations in~\eqref{eqn:Galerkin_O} and~\eqref{eqn:Galerkin_E}. To do that we also need to asymptotically expand $\vec{\alpha}$ and $\vec{\beta}$:
\begin{align}\label{eqn:asymptotic}
\vec{\alpha}^{\varepsilon,\delta} = \vec{\alpha}^{0,\delta} + \varepsilon\vec{\alpha}^{1,\delta} +\cdots\,,\\
\vec{\beta}^{\varepsilon,\delta} = \vec{\beta}^{0,\delta} + \varepsilon\vec{\beta}^{1,\delta} +\cdots\,.
\end{align}
We plug the expansion back into~\eqref{eqn:Galerkin_O} and~\eqref{eqn:Galerkin_E}, and match by the order of $\varepsilon$, then we get:
\begin{itemize}
\item{Leading order of~\eqref{eqn:Galerkin_O}:}
\begin{equation}\label{eqn:Galerkin_O_lead}
(\Phi\otimes\II)\cdot(\vec{\beta}^{0,\delta})^{n+1} = -\varepsilon(\Sigma^x\otimes\FF^\text{cos}+\Sigma^y\otimes\FF^\text{sin})\cdot(\vec{\alpha}^{0,\delta})^{n+1}\,.
\end{equation}
Then $(\vec{\beta}^{0,\delta})^{n+1} = \mathcal{O}(\varepsilon)$ is immediate if $(\vec{\alpha}^{0,\delta})^{n+1}$ can be shown to be at most $\mathcal{O}(1)$.
\item{Leading order of~\eqref{eqn:Galerkin_E}}
\begin{equation}\label{eqn:Galerkin_E_lead}
\left[\Sigma^\text{inv}\otimes(\II-\PP)\right]\cdot(\vec{\alpha}^{0,\delta})^{n+1} = 0\,.
\end{equation}
Since $\Sigma^\text{inv}$ is not singular this indicates that $(\vec{\alpha}^{0,\delta})^{n+1}$ is in the null space of $\II-\PP$. Considering the formula in~\eqref{eqn:BBPP} it is easy to see that:
\begin{equation}\label{eqn:Galerkin_E_lead_alpha}
\alpha^{0,\delta}_{m,n}=0\,,\quad\forall n\geq2\,, m\geq 1\,.
\end{equation}
Therefore (2) is shown, which indicates that the information in $\vec{\alpha}^{0,\delta}$ could be compressed into $\vec{\alpha}^{0,\delta}_1$.
\item{Applying $\PP$ on both sides of~\eqref{eqn:Galerkin_E}}
\begin{equation}\label{eqn:Galerkin_E_proj}
(\Phi\otimes\PP)\cdot(\vec{\alpha}^{0,\delta})^{n+1} = (\Phi\otimes\PP)\cdot(\vec{\alpha}^{0,\delta})^{n} - \frac{\Delta t}{\varepsilon}(\Xi^x\otimes\PP\FF^\text{cos}+\Xi^y\otimes\PP\FF^\text{sin})\cdot(\vec{\beta}^{0,\delta})^{n+1}\,.
\end{equation}
Here we have used the fact that $\PP$ is a projection and thus $\PP(\II - \PP) = 0$.
\end{itemize}
Combining~\eqref{eqn:Galerkin_O_lead} and~\eqref{eqn:Galerkin_E_proj}, one has:
\begin{equation*}
(\Phi\otimes\PP)\cdot(\vec{\alpha}^{0,\delta})^{n+1} = (\Phi\otimes\PP)\cdot(\vec{\alpha}^{0,\delta})^{n} + \Delta t[ (\Xi^x\cdot\Phi^{-1}\cdot\Sigma^x)\otimes(\PP\cdot\FF^\text{cos}\cdot\FF^\text{cos})+ (\Xi^y\cdot\Phi^{-1}\cdot\Sigma^y)\otimes(\PP\cdot\FF^\text{sin}\cdot\FF^\text{sin})]\cdot(\vec{\alpha}^{0,\delta})^{n+1}\,.
\end{equation*}

Due to~\eqref{eqn:Galerkin_E_lead_alpha}, all elements in $\vec{\alpha}^{0,\delta}$ diminish as $\varepsilon\to0$ except the $\vec{\alpha}^{0,\delta}_1$ term. For conciseness of the notation we denote it $\vec{\alpha}^\delta_1$. Considering
$$\PP\cdot\FF^\text{cos}\cdot\FF^\text{cos}=\PP\cdot\FF^\text{sin}\cdot\FF^\text{sin} = \frac{1}{2}\Id\,,$$
where $\Id$ stands for identity matrix, we compress all $n\geq 2$ terms in $\vec{\alpha}$ and have:
\begin{align*}
\Phi\cdot\vec{\alpha}^{\delta,n+1}_1 &= \Phi\cdot\vec{\alpha}^{\delta,n}_1 + \frac{1}{2}\Delta t\left[(\Xi^x\cdot\Phi^{-1}\cdot\Sigma^x)+(\Xi^y\cdot\Phi^{-1}\cdot\Sigma^y)\right]\cdot\vec{\alpha}^{\delta,n+1}_1\\
&=\Phi\cdot\vec{\alpha}^{\delta,n}_1 + \frac{1}{2}\Delta t\DD\cdot\vec{\alpha}^{\delta,n+1}_1\,,
\end{align*}
which shows (3). The convergence rate stated in (4) comes from the asymptotic expansion~\eqref{eqn:asymptotic}. To show it rigorously one also needs to write down the equation for $\vec{\alpha}^{1,\delta}$ and show the boundedness, which we will neglect the details here. 

Then (5) is obvious according to Lemma~\ref{lemma:limit_mat}, and the convergence rate stated in (6) comes from subtracting the two schemes~\eqref{eqn:alpha_1_delta} and~\eqref{eqn:limit_hom_dis}:
\begin{align*}
&\Phi_{\hom}\cdot(\vec{\alpha}^{n+1} - \vec{\alpha}^{\delta,n+1}) + (\Phi_{\hom} - \Phi)\cdot\vec{\alpha}^{\delta,n+1}\\
= & \Phi_{\hom}\cdot(\vec{\alpha}^{n} - \vec{\alpha}^{\delta,n}) + (\Phi_{\hom} - \Phi)\cdot\vec{\alpha}^{\delta,n} + \frac{\Delta t}{2} \DD_{\hom}\cdot(\vec{\alpha}^{n+1} - \vec{\alpha}^{\delta,n+1}) + \frac{\Delta t}{2}(\DD_{\hom} - \DD)\cdot\vec{\alpha}^{\delta,n+1}
\end{align*}
Assuming that $\vec{\alpha}^{\delta,n+1} = \mathcal{O}(1)$, it is clear the error cumulated is governed by $\DD_{\hom}-\DD$, which is of $\mathcal{O}(\sqrt{\delta})$.

Finally to show (7) that~\eqref{eqn:limit_hom_dis} is a consistent scheme for the limit heat equation. To do that we plug the exact solution to the homogenized heat equation~\eqref{eqn:limit_hom_heat} into the scheme. Suppose $u(t_n)$
is the true solution at time $t_n$ and $\vec{u}(t_n)$ the list of evaluation of $u(t_n)$ at the grid points, then:
\begin{align*}
&\Phi_{\hom}\cdot\vec{u}(t_{n+1}) - \Phi_{\hom}\cdot\vec{u}(t_n) - \frac{\Delta t}{2}\DD_{\hom}\cdot\vec{u}(t_{n+1})\\
=&\Phi_{\hom}\cdot \left[\vec{u}(t_{n})+\Delta t\partial_t\vec{u}(t_n)+\cdots\right] - \Phi_{\hom}\cdot\vec{u}(t_n) - \frac{\Delta t}{2}\DD_{\hom}\cdot \left[\vec{u}(t_{n})+\Delta t\partial_t\vec{u}(t_n)+\cdots\right]\\
=&\Delta t\Phi_{\hom}\cdot\left[\partial_t\vec{u}(t_n)-\frac{1}{2}\Phi^{-1}_{\hom}\DD_{\hom}\vec{u}(t_n)\right] + \mathcal{O}((\Delta t)^2)\,.
\end{align*}
By Lemma~\ref{lemma:1st_derivative} and Remark~\ref{rmk:derivatives}, it is shown that $\Phi^{-1}_{\hom}\DD_{\hom}$ presents an $\mathcal{O}(H^2)$ approximation to $\nabla\cdot\left(a_{\hom}\nabla\right)$ given $u\in C^2(\rd\xx)$, and this leads to the final error term:
\begin{align*}
&\Phi_{\hom}\cdot\vec{u}(t_{n+1}) - \Phi_{\hom}\cdot\vec{u}(t_n) - \frac{\Delta t}{2}\DD_{\hom}\cdot\vec{u}(t_{n+1})\\
=&\Delta t\Phi_{\hom}\cdot\left[\partial_tu(t_n)-\frac{1}{2}\nabla\cdot\left(a_{\hom}\nabla u(t_n)\right)\right] + \Delta t^2 + \Delta tH^2\\
=& \mathcal{O}(\Delta tH^2+(\Delta t)^2)\,.
\end{align*}
This indicates that the cumulative $\mathcal{O}(1)$ time truncation error is $\mathcal{O}(H^2+ \Delta t)$. Stability is immediate due to the implicit time discretization, and this finishes the proof for the theorem.
\end{proof}

\begin{remark}\label{rmk:conv_1D}
To show (4), we used the fact that $\DD_{\hom}-\DD=\mathcal{O}(\sqrt{\delta})$ in higher dimension. This could be imporved to $\DD_{\hom}-\DD=\mathcal{O}(\delta)$ in 1D.
\end{remark}

\begin{remark}\label{rmk:asym_limit}
As stated in Remark~\ref{rmk:asym_formulation}, it is possible to keep the $a^\delta$ on the left side of the even equation, but if we follow the proof shown above, an asymmetric formulation will be obtained for the limiting heat equation. Indeed, Equation~\eqref{eqn:Galerkin_O_lead} is kept, and Equation~\eqref{eqn:Galerkin_E_lead} will be changed to:
\begin{equation}
\left[\Phi\otimes(\II-\PP)\right]\cdot\vec{\alpha}^{n+1} = 0\,,
\end{equation}
which leads to the same conclusion that $\vec{\alpha}^{n+1}$ is in the null space of $\II-\PP$. In order to close the system, instead of applying $\PP$ onto~\eqref{eqn:Galerkin_E}, we do so to~\eqref{eqn:Galerkin_asym_E}:
\begin{equation}
\left(\Sigma\otimes\PP\right)\cdot\vec{\alpha}^{n+1}=\left(\Sigma\otimes\PP\right)\cdot\vec{\alpha}^n-\left\{\frac{\Delta t}{\varepsilon}\Sigma^x\otimes\PP\FF^{\text{cos}}+\frac{\Delta t}{\varepsilon}\Sigma^y\otimes\PP\FF^{\text{sin}}\right\}\cdot\vec{\beta}^{n+1}\,.
\end{equation}
and plugging in~\eqref{eqn:Galerkin_O_lead}, one has:
\begin{align*}
\Sigma\cdot\vec{\alpha}^{n+1}_1 &= \Sigma\cdot\vec{\alpha}^{n}_1 + \Delta t\left[(\Sigma^x\cdot\Phi^{-1}\cdot\Sigma^x)+(\Sigma^y\cdot\Phi^{-1}\cdot\Sigma^y)\right]\cdot\vec{\alpha}^{n+1}_1\,,
\end{align*}
Note that according to the definition of $\Sigma^x_{mn} = \langle\phi_m\,,a^\delta \partial_x\phi_n\rangle_\xx$, it is not a symmetric matrix. The scheme above is roughly speaking a discretization for:
\begin{equation*}
a^\delta \partial_t\rho = a^\delta \nabla_\xx\cdot(a^\delta\nabla_\xx\rho)\,.
\end{equation*}
Aside from the fact that MsFEM convergence is unknown to this
equation, the numerical solution also fails to respect the symmetry,
which is undesirable. 
\end{remark}

\section{Numerical example}
In this section we report several numerical tests to 
demonstrate the effectiveness of our method for transport equation
with multiscale scattering coefficient.

\subsection{1D example}
In the first example we set the media as:
\begin{equation*}
a = 1.1+\sin{(10\pi x)}\,,\quad\text{with}\quad x\in[-1,1]
\end{equation*}
The media is periodic with ten periods in the domain $[-1,1]$. We first check the consistency. We compute the equation with $N_x$ set as $50$, $100$ and $200$ respectively, and we do not observe difference in the numerical solution. We note here that having $N_x = 50$ means putting five points in one period and the solution looks very under-resolved, but still at the discrete point, the numerical solution very well captures the result given by a much more resolved $N_x=200$ case. We also test such consistency on the transport equation level and numerically we observe that by setting $N_x=100$ and $N_x=200$ we obtain the same solution.

\begin{figure}[htbp]
\centering
\includegraphics[width = 0.45\textwidth]{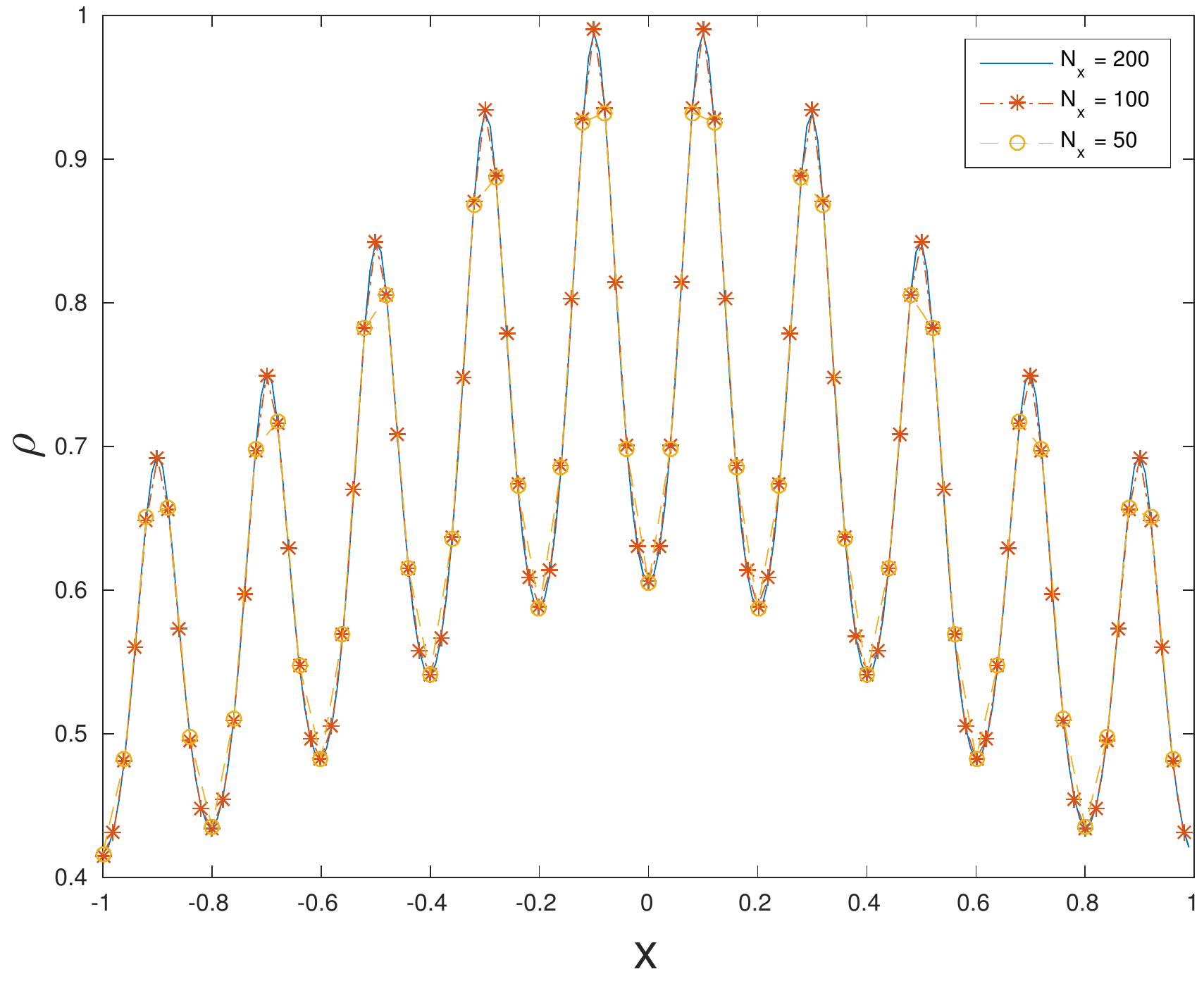}
\includegraphics[width = 0.45\textwidth]{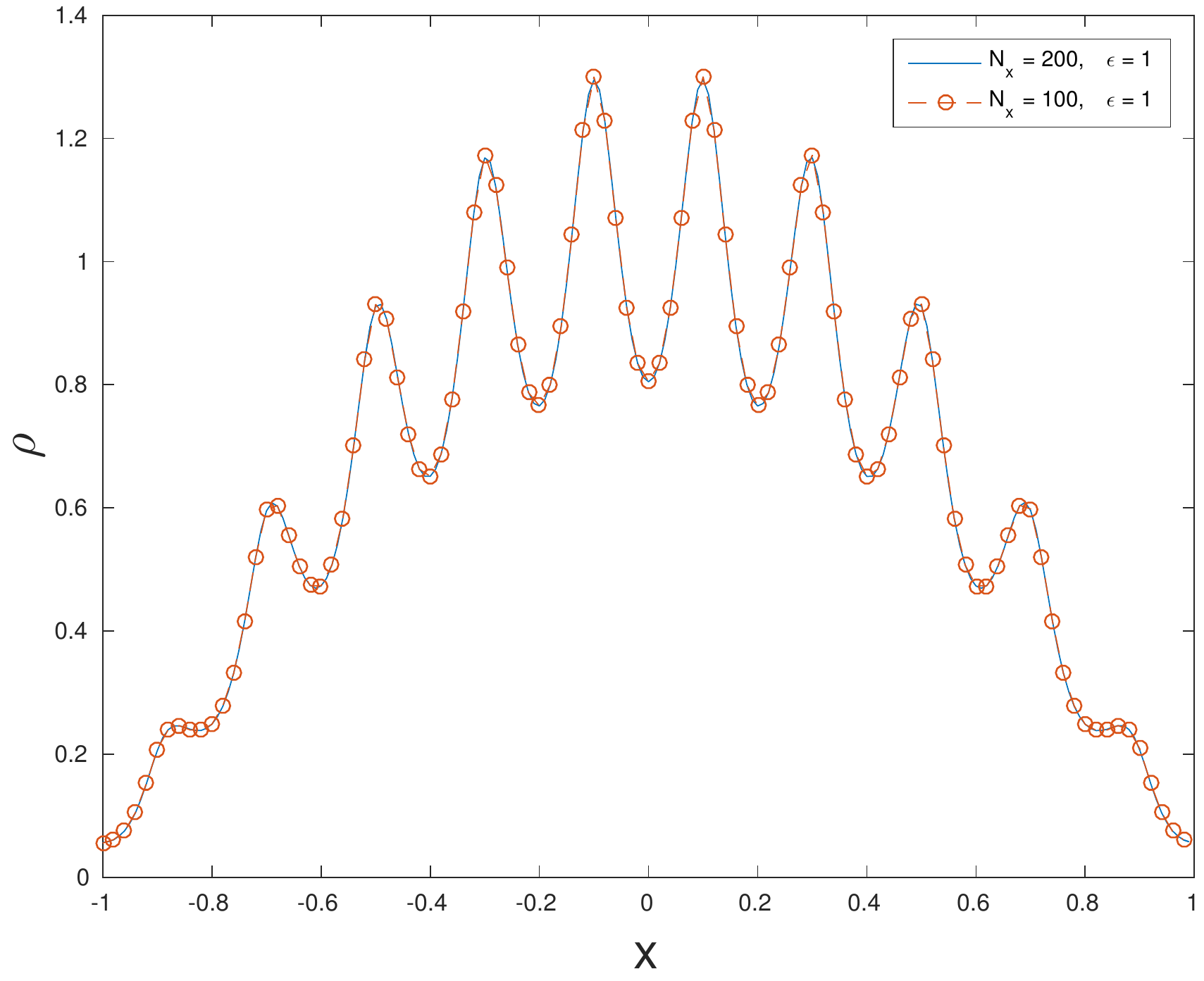}
\caption{The numerical solution is captured well with under-resolved grid points for both heat limit and the transport equation. The transport equation (shown on the right) is computed with $\varepsilon = 1$.}
\end{figure}

Then we test the convergence towards the heat limit with the Knudsen
number converging to zero. Numerical solution provided by the
under-resolved scheme for the transport equation converges to the
resolved numerical solution to the heat limit, see in Figure~\ref{fig:1D_10pi_eps_limit}.

\begin{figure}[htbp]
\centering
\includegraphics[width = 0.65\textwidth]{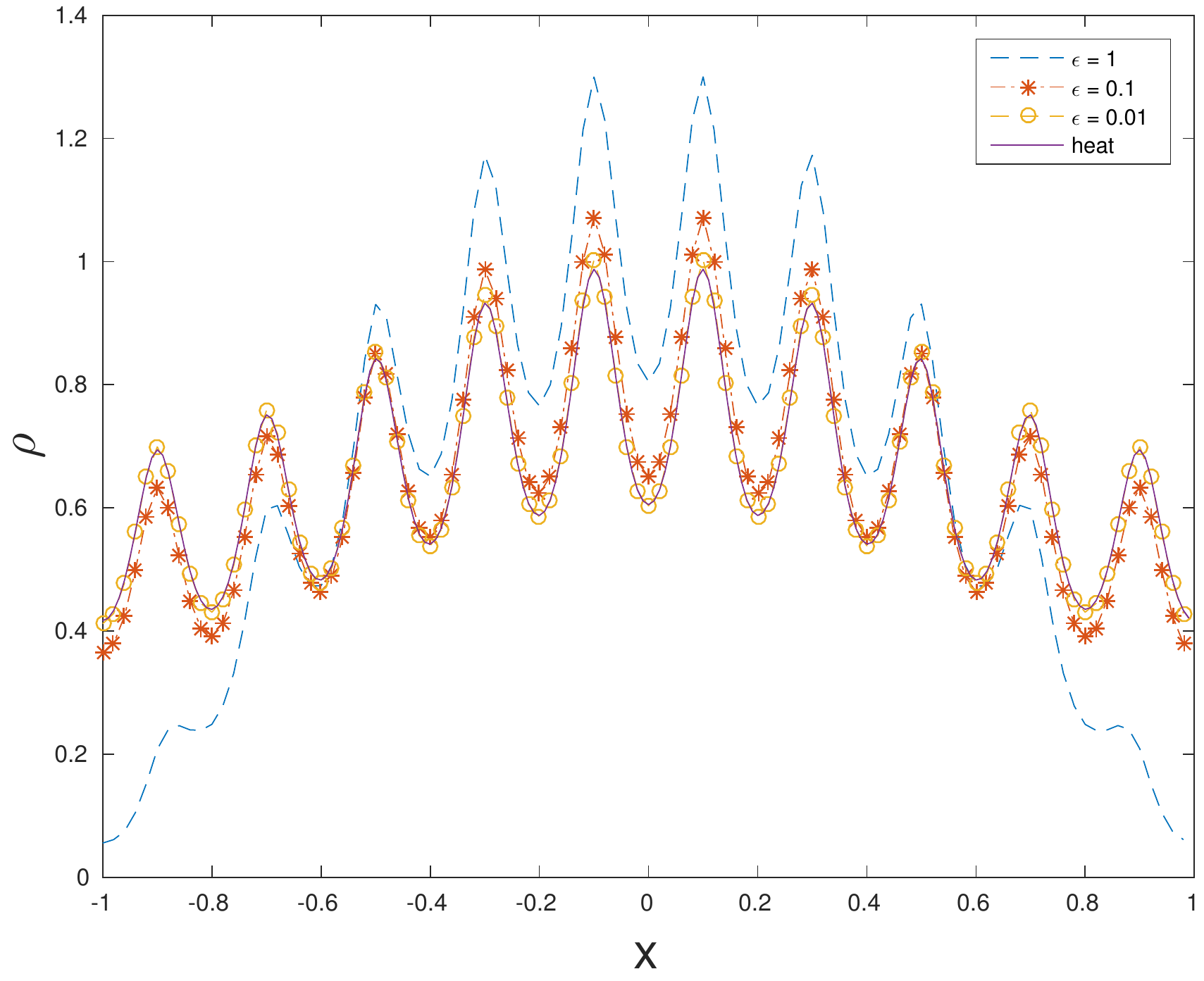}
\caption{With shrinking Knudsen number $\varepsilon$, the numerical solution for the transport equation converges to that to of the heat equation. Here the heat equation is computed with full-resolved mesh while for the computation of the transport equation, we use $N_x=100$ mesh points.}\label{fig:1D_10pi_eps_limit}
\end{figure}

The same process is repeated with a even more challenging media set as:
\begin{equation*}
a = 1.1+\sin{(20\pi x)}\,,\quad\text{with}\quad x\in[-1,1]
\end{equation*}
Here the oscillation in the media is even stronger. We put $20$ periods in the domain of $[-1,1]$. The consistency is plotted in Figure~\ref{fig:1D_20pi_resolve} where we show numerical results provided by setting $N_x$ equal to $50$, $100$ and $200$ respectively. On average with $50$ grid points, one period gets about $2.5$ grid points and the solution is well below being resolved. Still we see the numerical solution is captured very well at the discrete points. Same consistency is observed numerically for the transport equation as well. The convergence of under-resolved transport equation computed with our method towards the heat limit is demonstrated in Figure~\ref{fig:1D_20pi_limit}.

\begin{figure}[htbp]
\centering
\includegraphics[width = 0.45\textwidth]{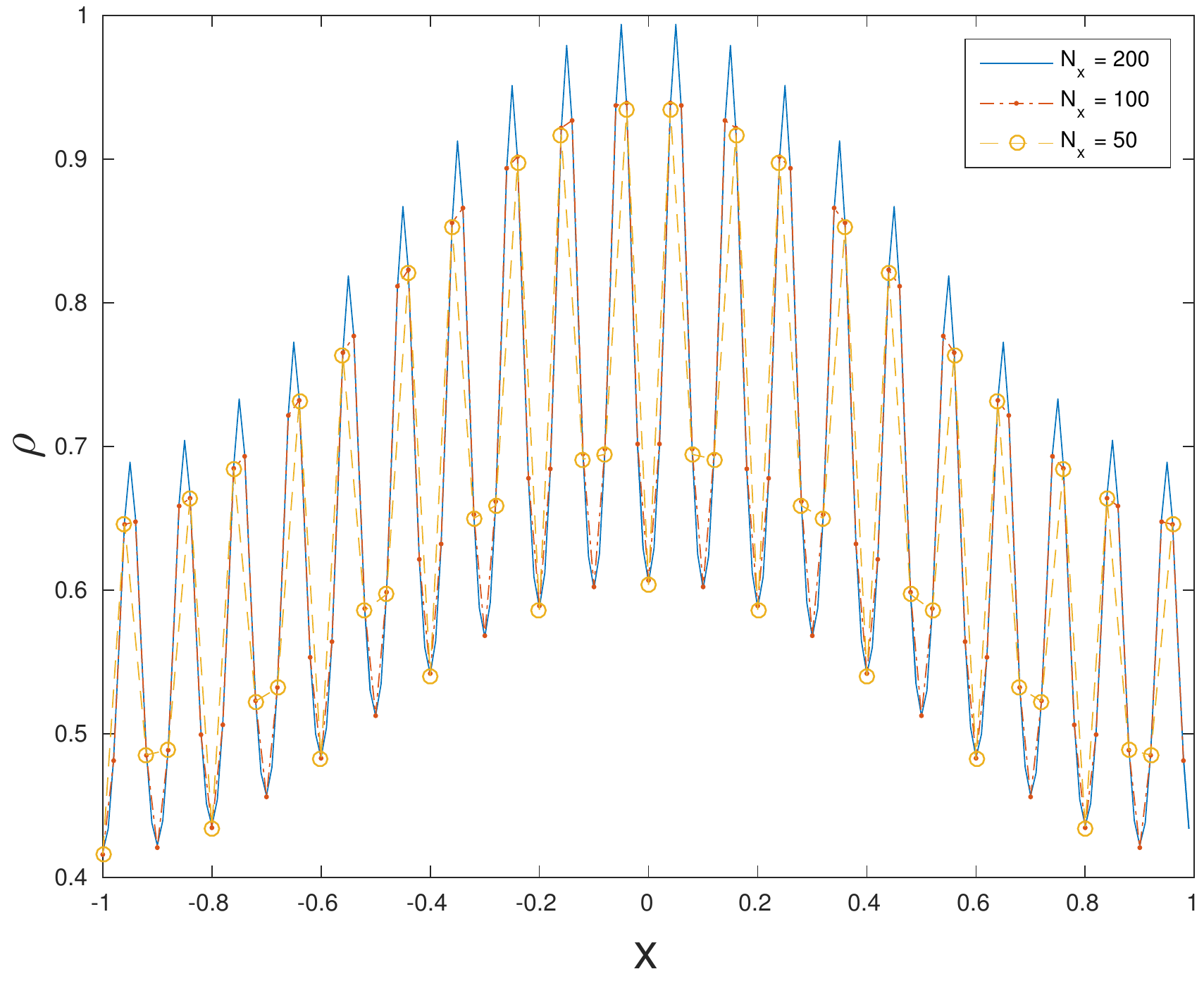}
\includegraphics[width = 0.45\textwidth]{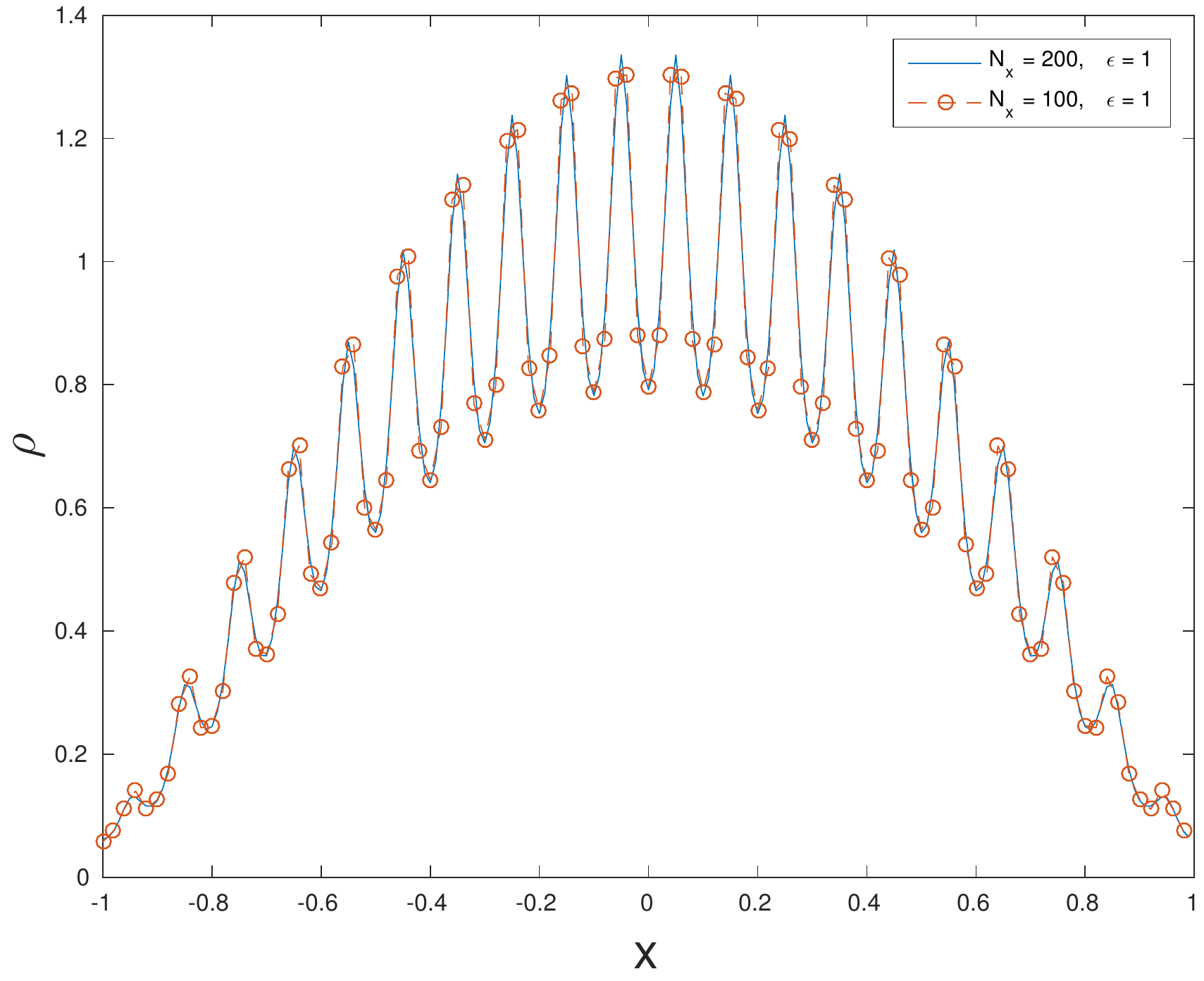}
\caption{The numerical solution is captured well with under-resolved grid points for both heat limit and the transport equation. The transport equation (on the right) is computed with $\varepsilon = 1$.}\label{fig:1D_20pi_resolve}
\end{figure}

\begin{figure}[htbp]
\centering
\includegraphics[width = 0.47\textwidth,height = 2.5in]{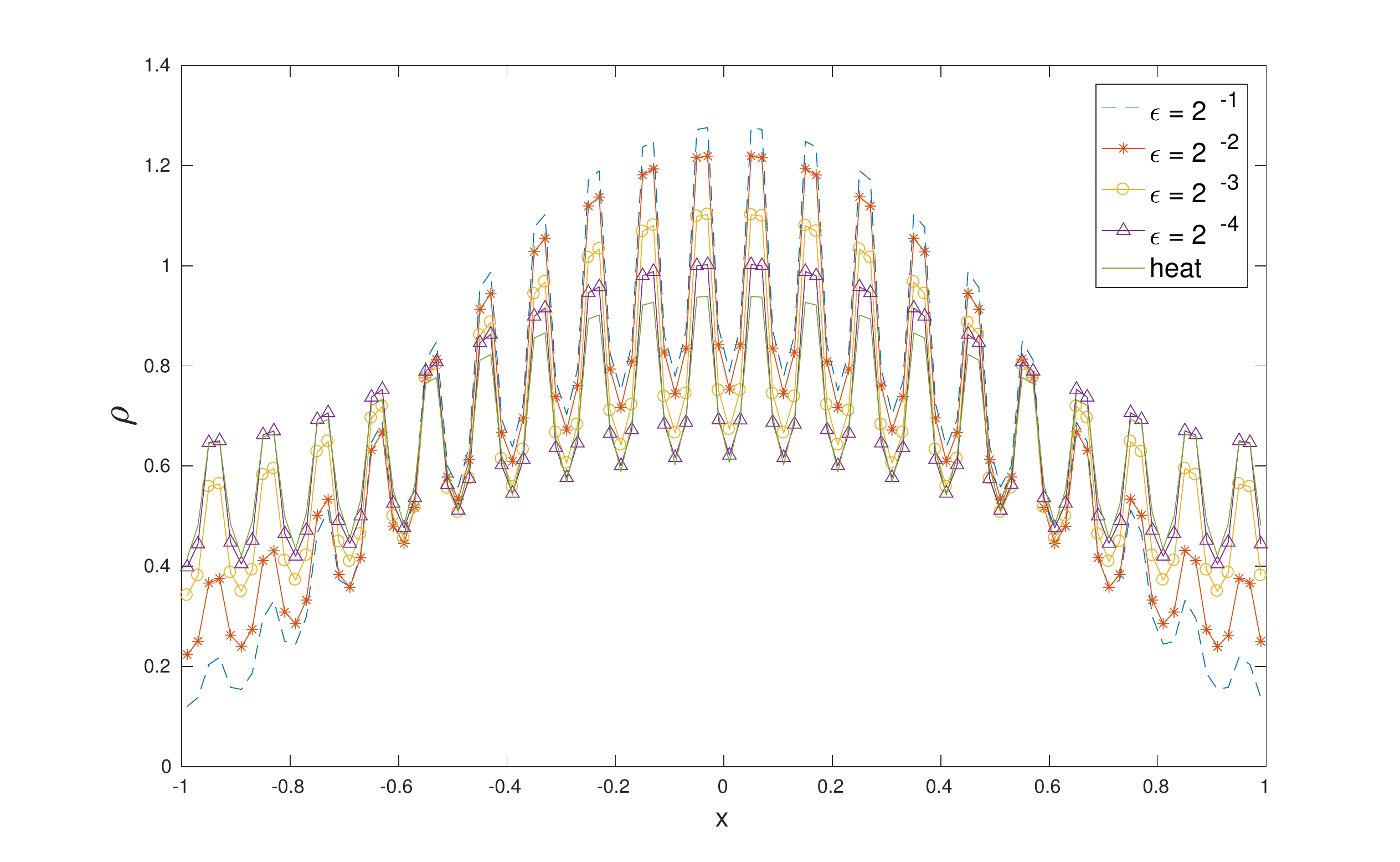}
\includegraphics[width = 0.5\textwidth,height = 2.5in]{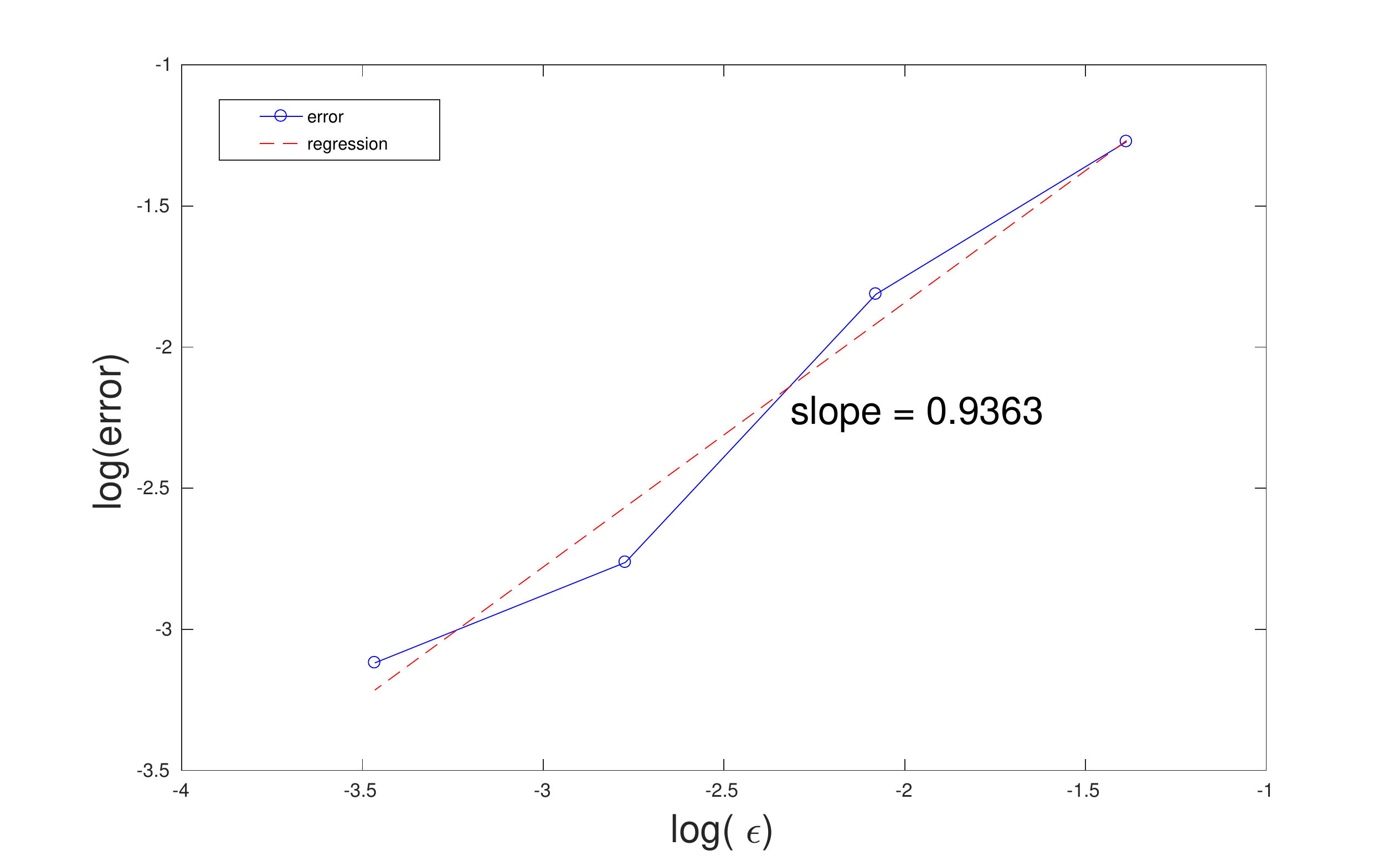}
\caption{With shrinking Knudsen number $\varepsilon$, the numerical solution for the transport equation converges to that to of the heat equation. The convergence rate is about $\mathcal{O}(\varepsilon)$. Here the heat equation is computed with full-resolved mesh while for the computation of the transport equation, we use the under-resolved $N_x=100$ mesh points.}\label{fig:1D_20pi_limit}
\end{figure}

\subsection{1D example, convergence in $\delta$}
In this example we check the convergence in $\delta$. Since we are in 1D, the analytical homogenized coefficient could be explicitly expressed, and the error, according to Theorem~\ref{thm:diff_limit} and Remark~\ref{rmk:conv_1D} should be of order $\mathcal{O}(\delta)$. In the domain $[-1,1]$, we set the media as:
\begin{equation}
a(x) = \frac{1}{\cos(2\pi x/\delta) + 4}
\end{equation}
and thus $a^\ast = \frac{1}{4}$. The regime being studied in this paper requires $\varepsilon\ll \delta$, and thus we choose $\delta = [1/8,1/24,1/40,1/56]$ and $\varepsilon = 2^{-10}$. $H = 1/32$ while $h=1/1280$. We plot the solution with various $\delta$ at $T = 0.1$ together with the solution to the homogenized heat solution in Figure~\ref{fig:1D_conv_compare} and we also show the convergence rate.
\begin{figure}[htbp]
\centering
\includegraphics[width = 0.47\textwidth,height = 2.5in]{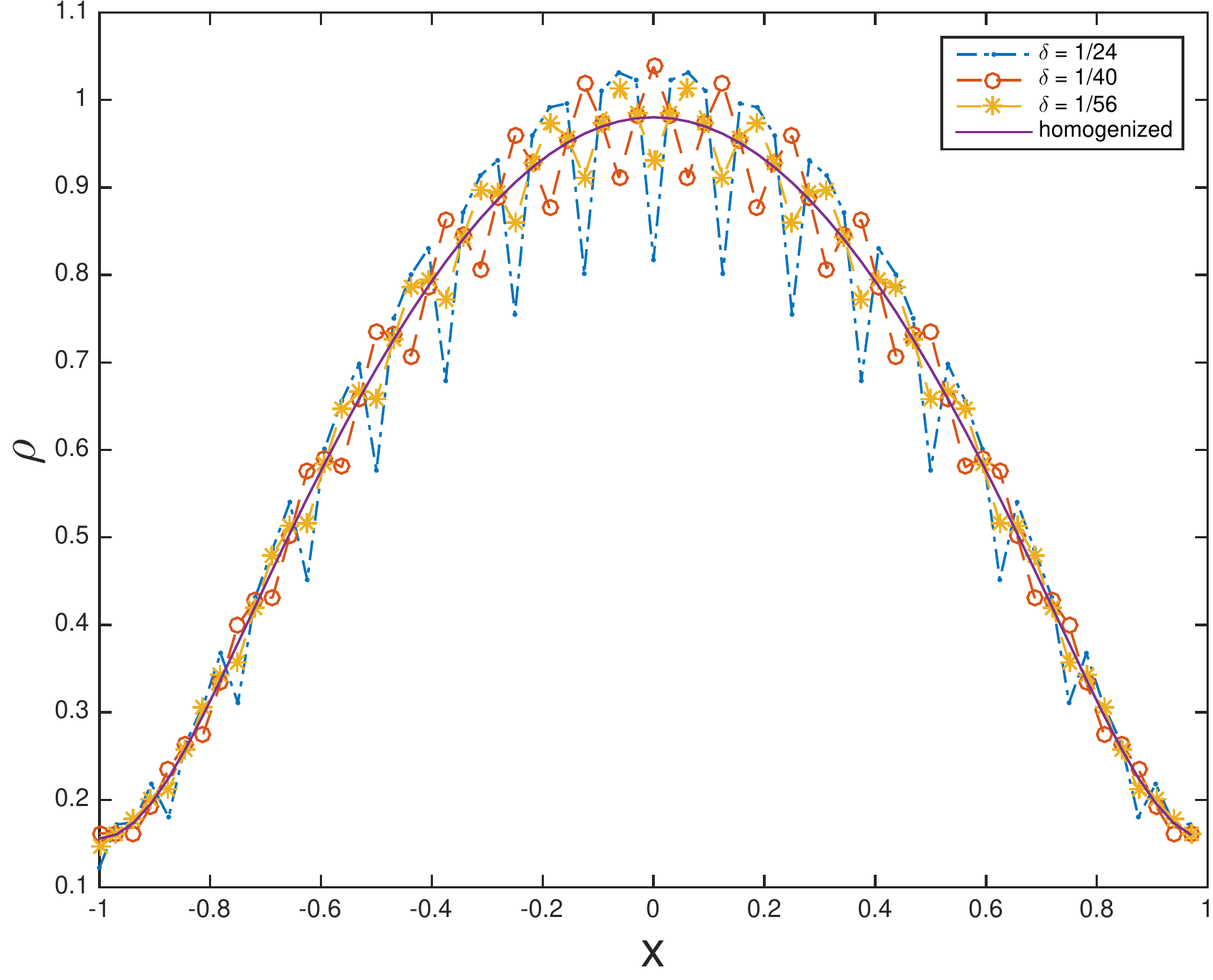}
\includegraphics[width = 0.47\textwidth,height = 2.5in]{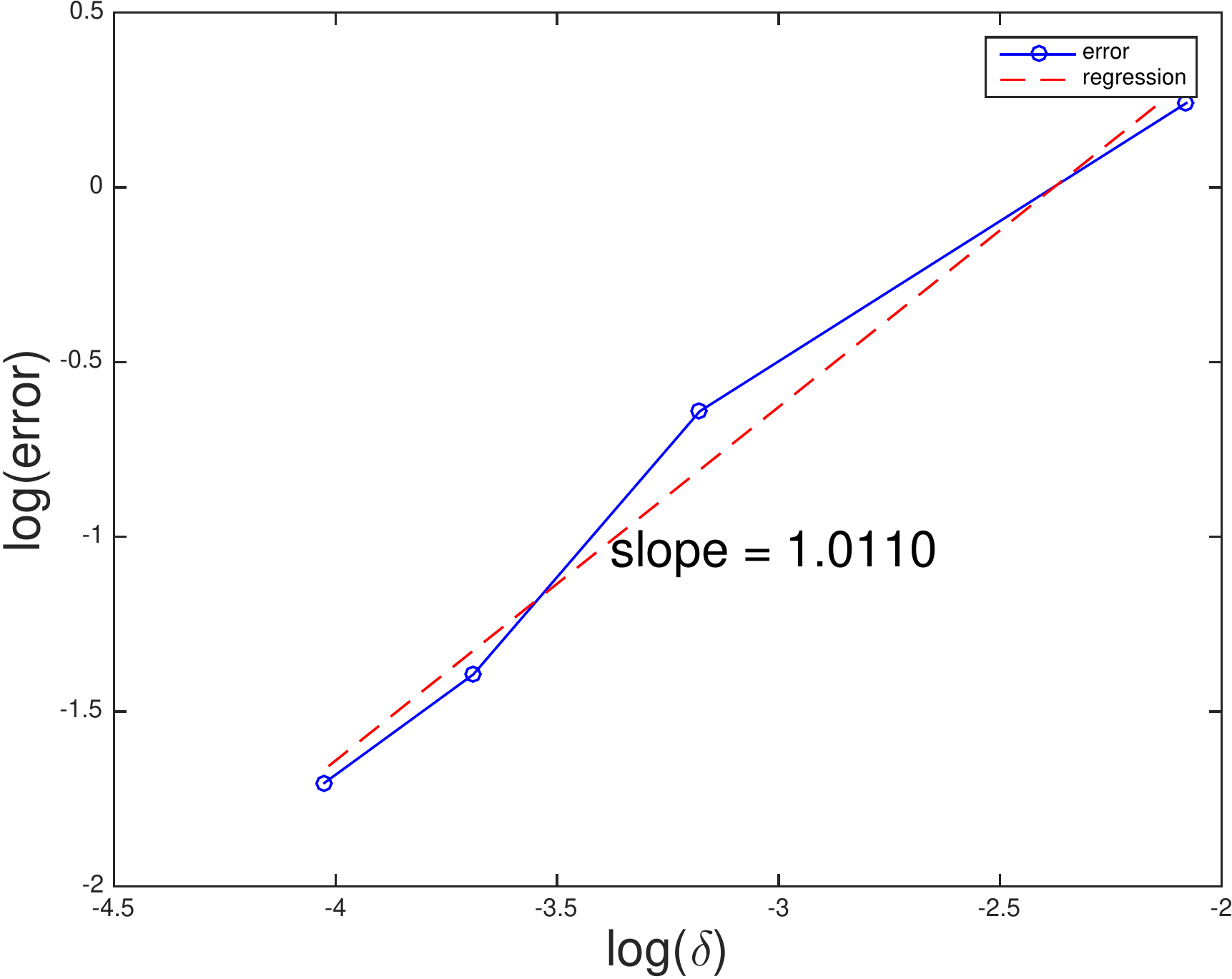}
\caption{The plot on the left panel shows the solution to $\rho$ with different $\delta$ compared with the homogenized heat equation limit. On the right we plot the convergence rate in terms of $\delta$. It shows $\mathcal{O}(\delta)$ convergence. This is aligned with our prediction in 1D.}\label{fig:1D_conv_compare}
\end{figure}

\subsection{2D example}
In the third example we check the solution behavior in 2D. We still have periodic media set as:
\begin{equation*}
a(x,y) = 1.1+\sin{(2\pi x)}\sin{(10\pi y)}\,,\quad\text{with}\quad (x,y)\in[-1,1]\times[-1,1]\,.
\end{equation*}
Here the oscillation along $y$ is much heavier than that in $x$. The media is plotted in Figure~\ref{fig:2D_media}.

\begin{figure}[htbp]
\centering
\includegraphics[width = 0.45\textwidth]{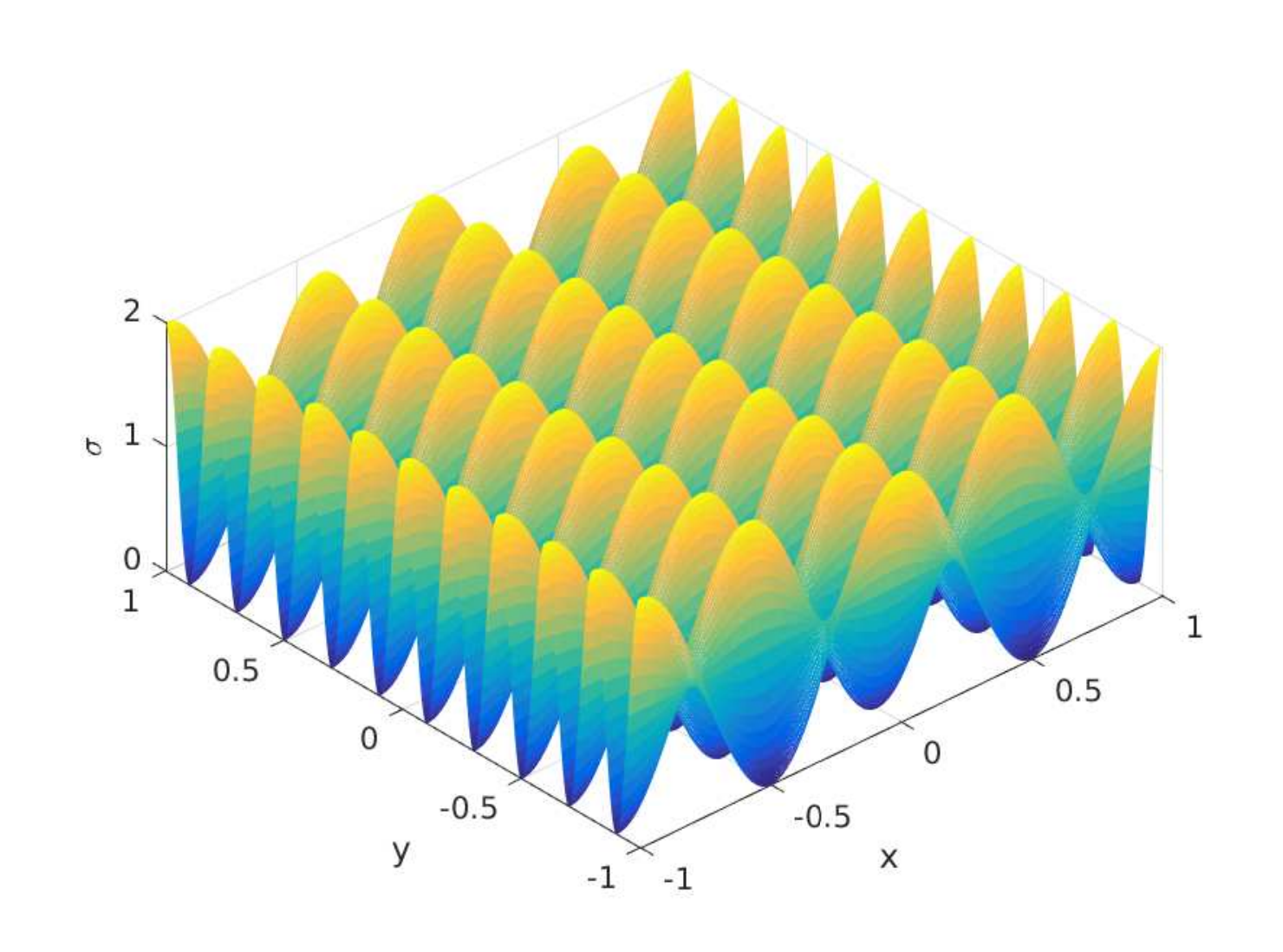}
\caption{The media is periodic but highly oscillatory in both directions and the oscillation in $y$ direction is very strong. We put 10 periods along $y$.}\label{fig:2D_media}
\end{figure}

To test the consistency of the method we compute the heat equation limit with $N=50$ and $N=100$. The differences between the two solutions are negligible, as shown in Figure~\ref{fig:2D_resolve}.

\begin{figure}[htbp]
\centering
\includegraphics[width = 0.85\textwidth,height = 0.25\textheight]{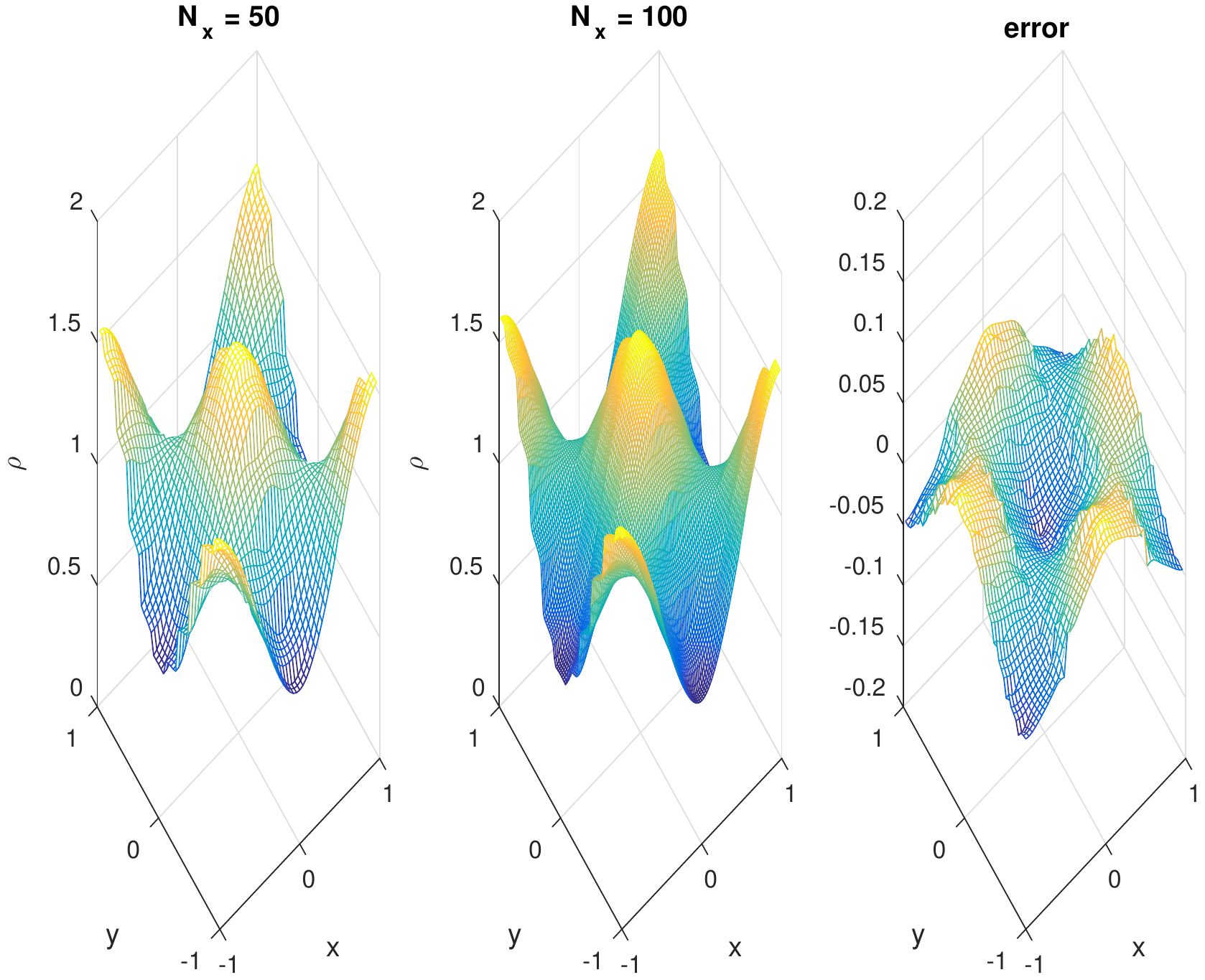}
\caption{Compare the numerical solution to the heat limit equation using $N_x = 50$ and $N_x=100$. Here by setting $N_x=50$ we obtain an under-resolved solution. The error shown in the right panel suggest $0.05$ out of $1.5$ error.}\label{fig:2D_resolve}
\end{figure}

In Figure~\ref{fig:2D_error} and~\ref{fig:2D_limit} we show the convergence of the transport equation with highly oscillatory media towards the heat limit with the same oscillatory media.

\begin{figure}[htbp]
\centering
\includegraphics[width = 0.85\textwidth,height = 0.25\textheight]{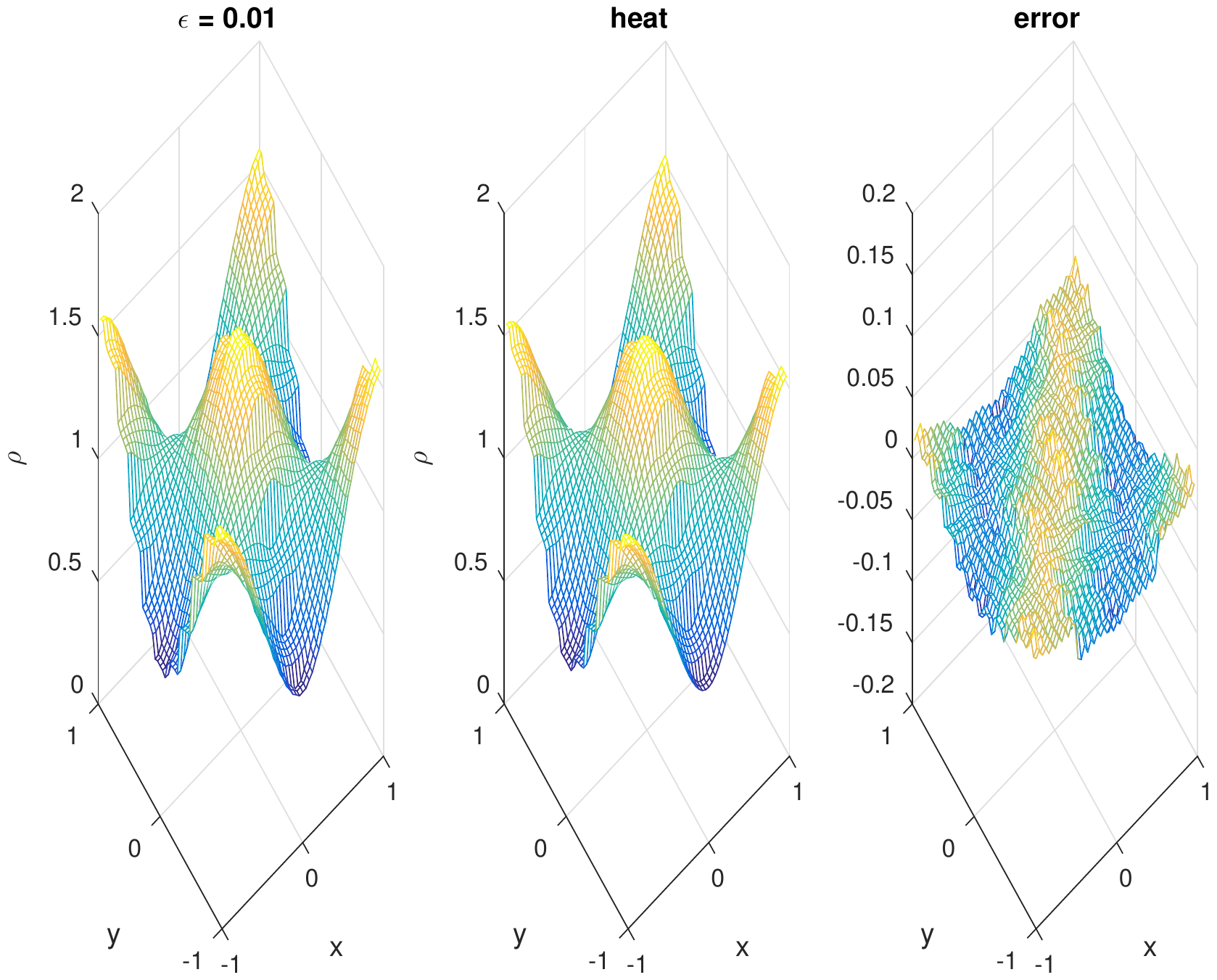}
\caption{We compare the solution to the transport equation with $\varepsilon = 0.01$ (left panel) with the solution to the heat limit (middle panel), and it is very well captured. The error is plotted in the right panel.}\label{fig:2D_error}
\end{figure}

\begin{figure}[htbp]
\centering
\includegraphics[width = 0.52\textwidth,height = 0.32\textheight]{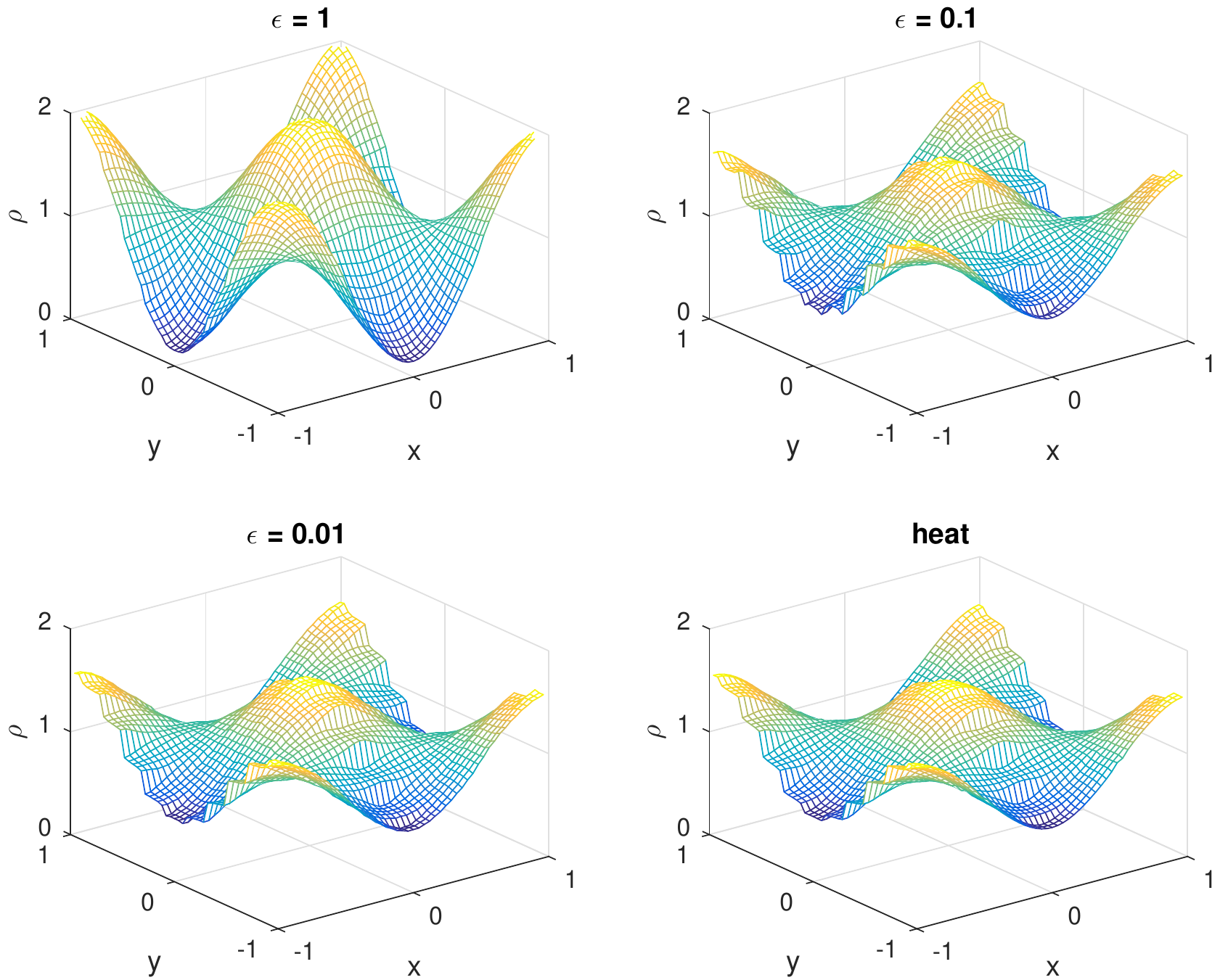}\\
\includegraphics[width = 0.32\textwidth,height = 0.25\textheight]{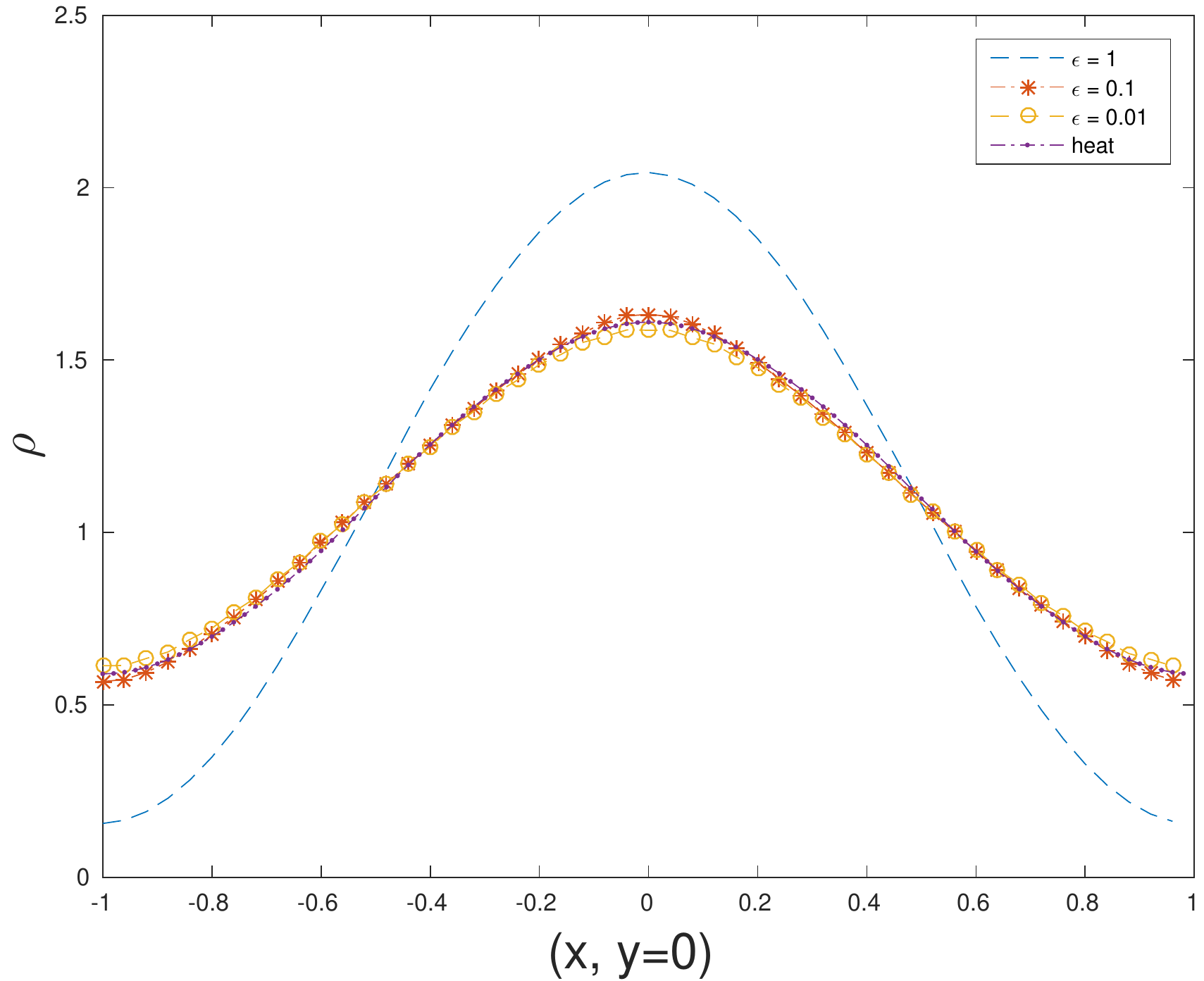}
\includegraphics[width = 0.32\textwidth,height = 0.25\textheight]{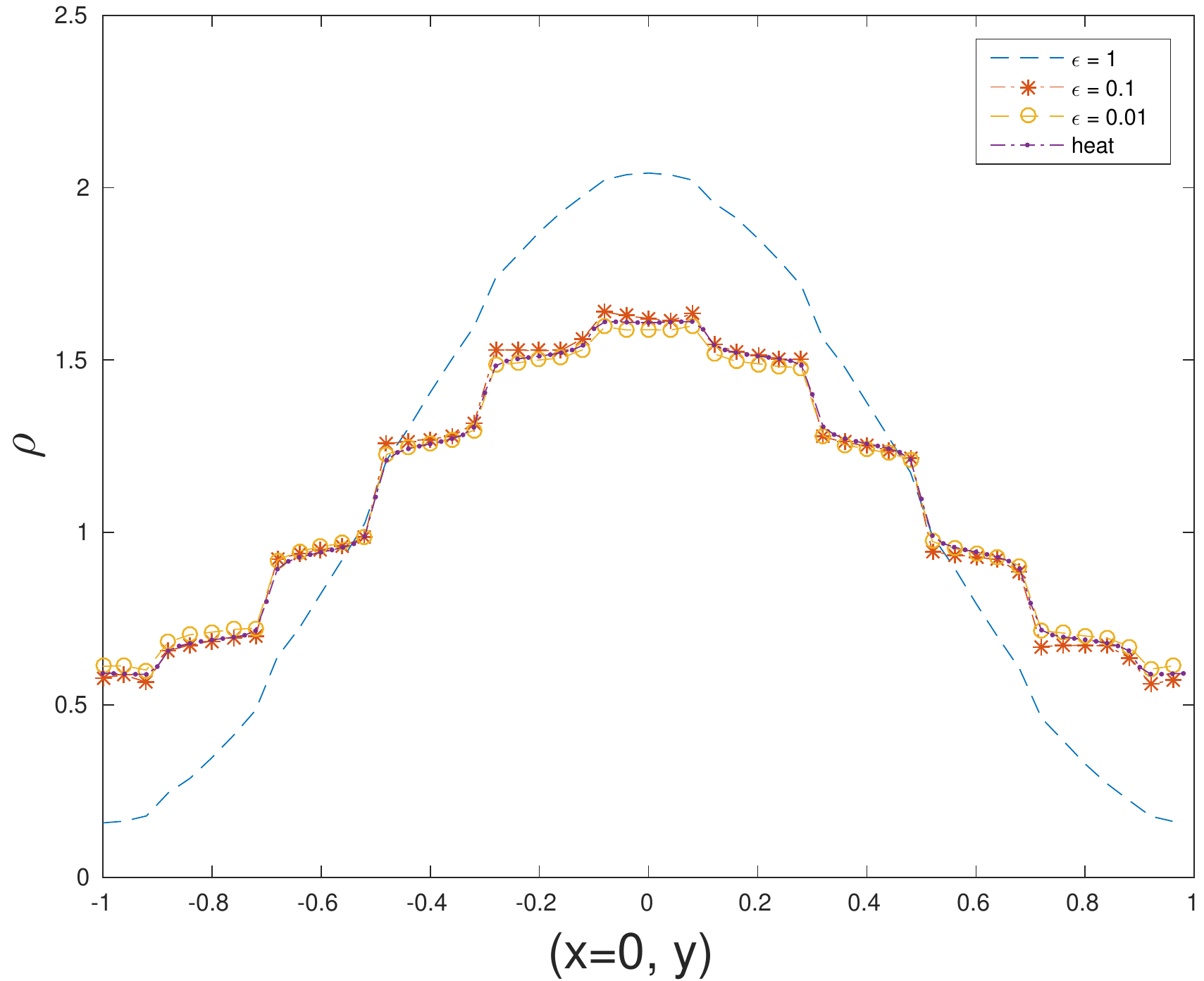}
\caption{With shrinking Knudsen number $\varepsilon$, the numerical solution for the transport equation gradually converges to that to of the heat equation. We also plot two intersections: the middle panel we show the solution with along $x$ direction at $y=0$ and the right panel shows the comparison of the solution along $y$ direction at $x=0$. The solution along $x$ is smooth since the oscillation is not as strong. Along $y$ direction, the solution experience some big jumps but they are all captured well.}\label{fig:2D_limit}
\end{figure}

\subsection{Benchmark example for symmetric and asymmetric formulations}
In this example we adopt the media used in~\cite{EHW:00,HW_numerics:99}:
\begin{equation*}
a(x,y) = \frac{2+1.8\sin{(10\pi x)}}{2+1.8\cos{(10\pi y)}}+\frac{2+\sin{(10\pi y)}}{2+1.8\sin{(10\pi x)}}\,.
\end{equation*}
The media is plotted in Figure~\ref{fig:2D_bench_media}.

\begin{figure}[htbp]
\centering
\includegraphics[width = 0.85\textwidth,height = 0.25\textheight]{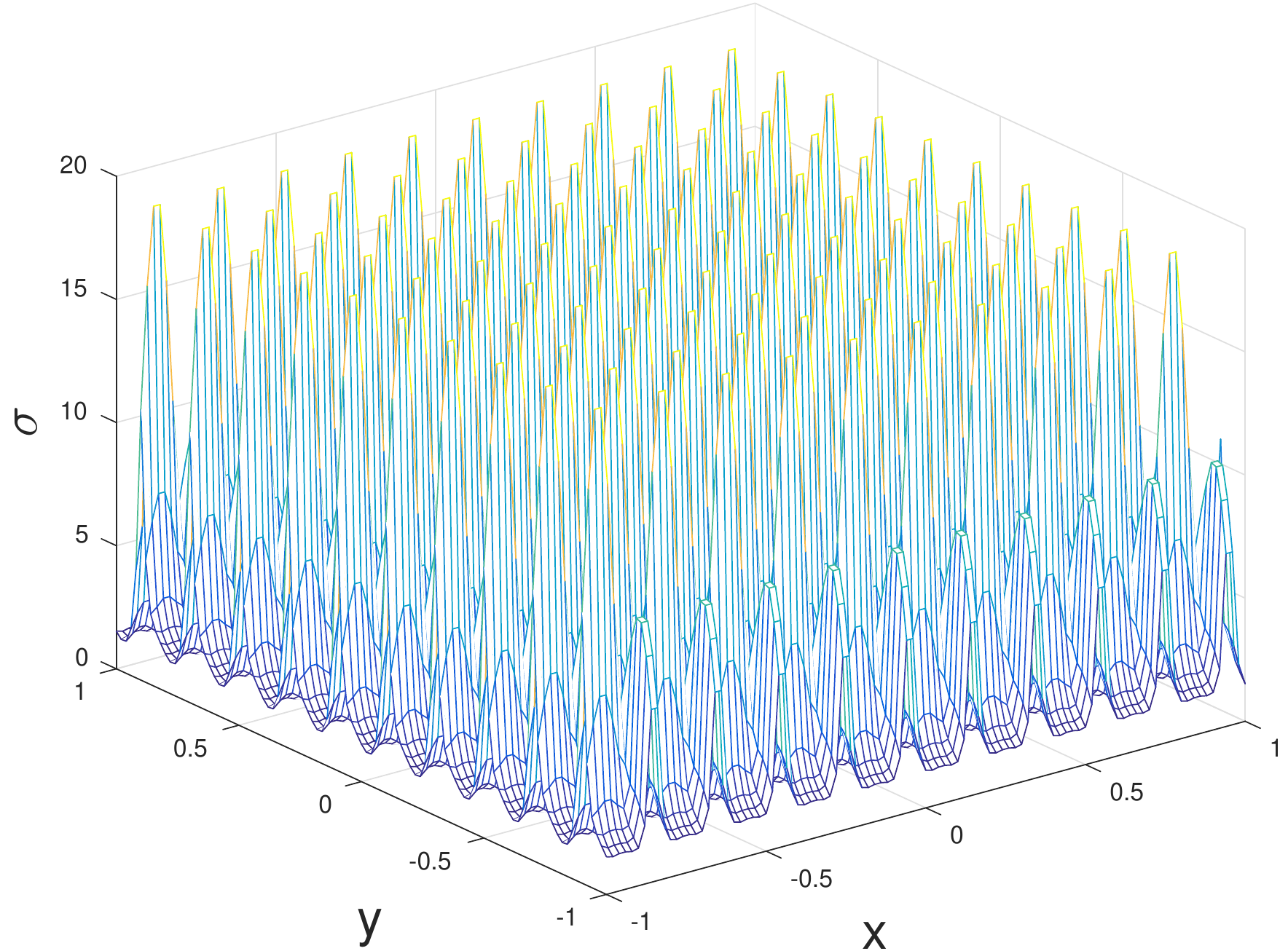}
\caption{Media is adopted from~\cite{EHW:00,HW_numerics:99}. It is periodic and highly oscillatory in both directions.}\label{fig:2D_bench_media}
\end{figure}

As mentioned in Remark~\ref{rmk:asym_limit} before, the asymmetric formulation of the heat equation is not preferred, we here plot the solution to the limiting heat equation using both the symmetric and asymmetric formulation. Compare the two numerical results shown in Figure~\ref{fig:2D_bench_heat_resolution}, it is obvious the asymmetric formulation failed to maintain the symmetric solution profile. We then demonstrate the AP property in Figure~\ref{fig:2D_bench_AP_comparison}, where the error obtained using different $\varepsilon$ is shown. We also show the convergence of the solution to the transport equation to that of the heat equation on two intersections ($(x=0,y)$ and $(x,y=0)$).

\begin{figure}[htbp]
\centering
\includegraphics[width = 0.85\textwidth,height = 0.25\textheight]{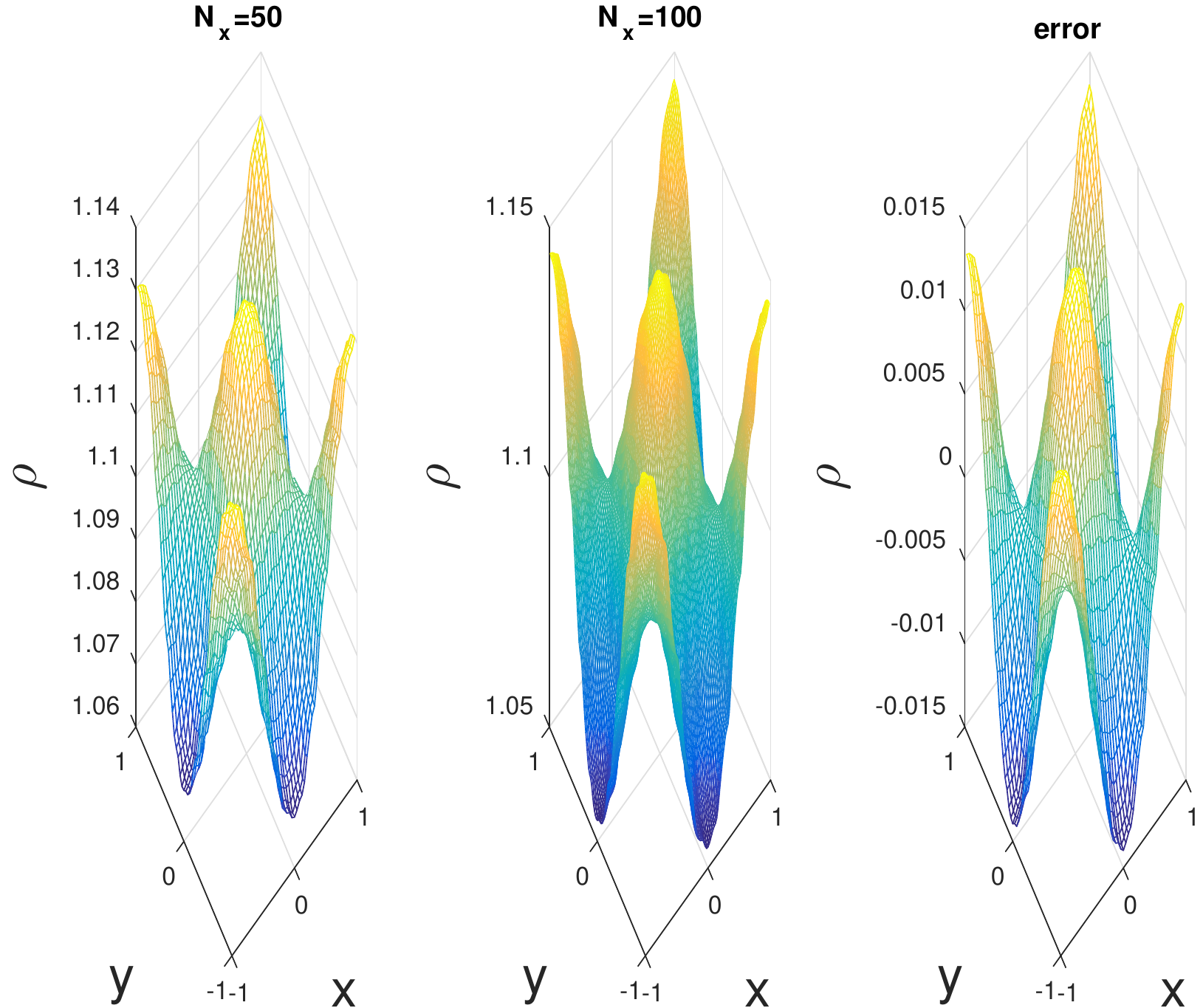}
\includegraphics[width = 0.85\textwidth,height = 0.25\textheight]{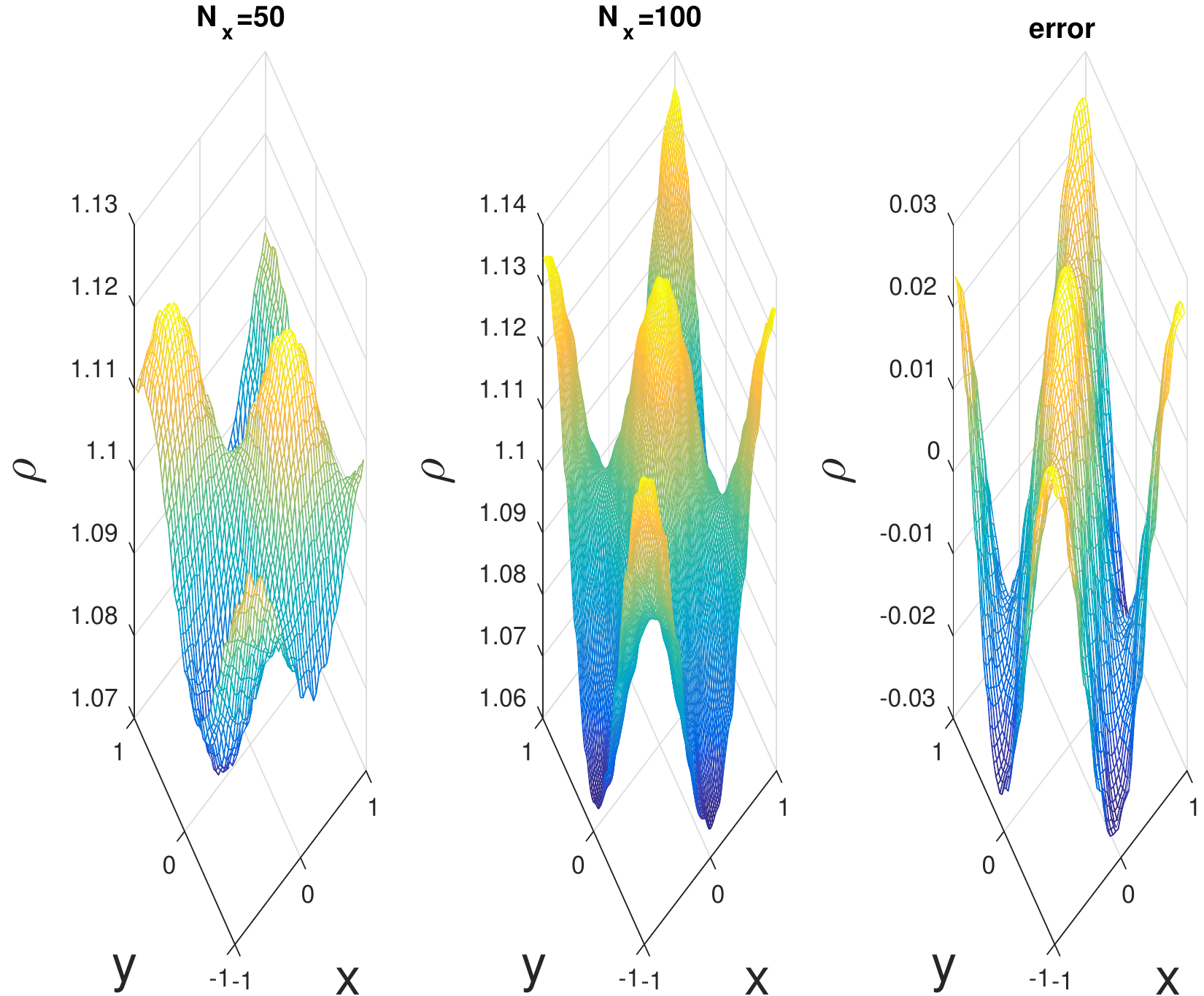}
\caption{The top three figures show the comparison of the solutions computed using coarse and fine resolution given by the symmetric formulation, and the bottom three show the comparison of the solution using the asymmetric formulation. The asymmetric profile is obvious in the bottom left figure which shows that using the asymmetric formulation the numerical result fails to capture the symmetric profile in the solution.}\label{fig:2D_bench_heat_resolution}
\end{figure}

\begin{figure}[htbp]
\centering
\includegraphics[width = 0.85\textwidth,height = 0.25\textheight]{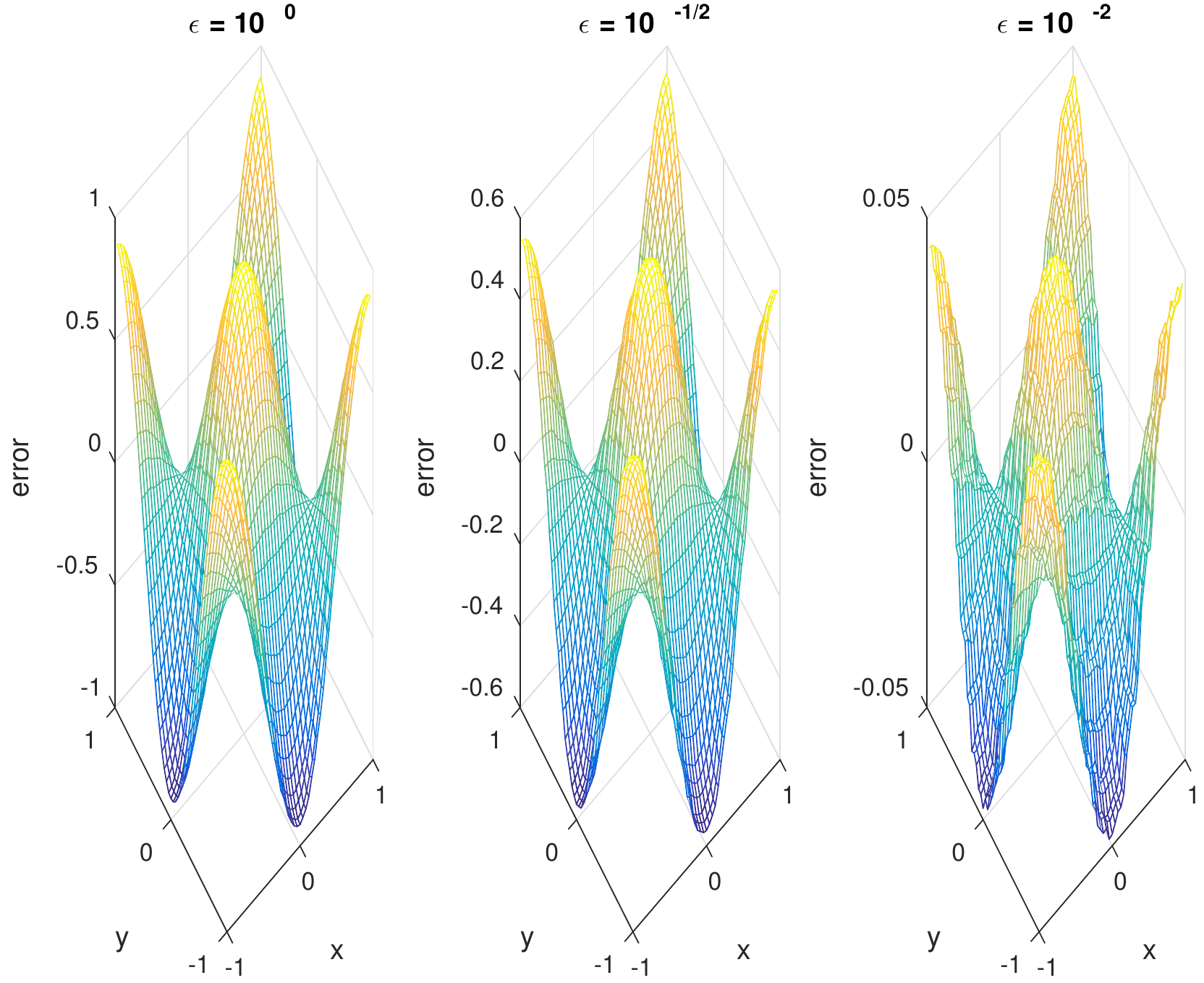}
\includegraphics[width = 0.45\textwidth,height = 0.25\textheight]{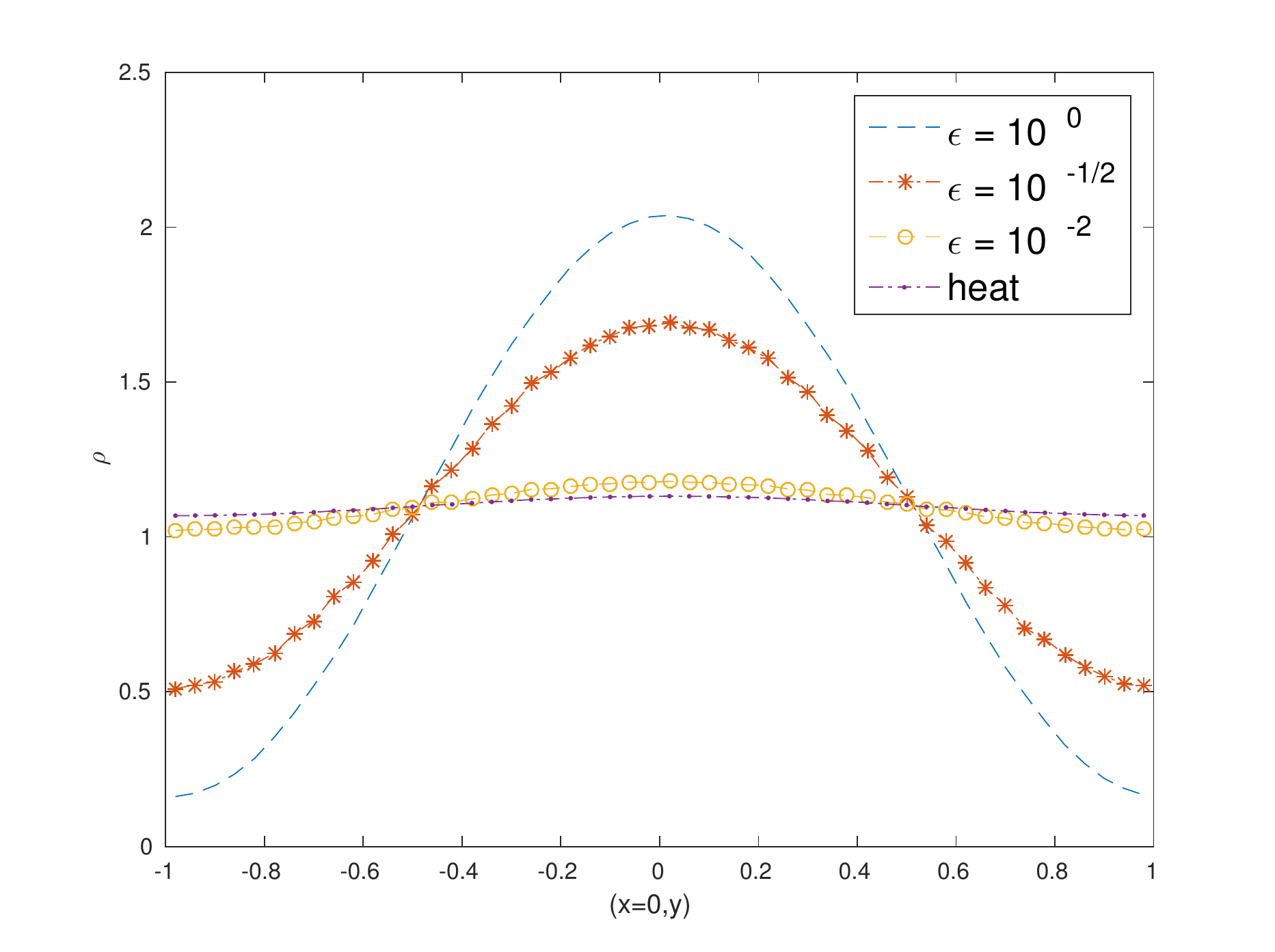}
\includegraphics[width = 0.45\textwidth,height = 0.25\textheight]{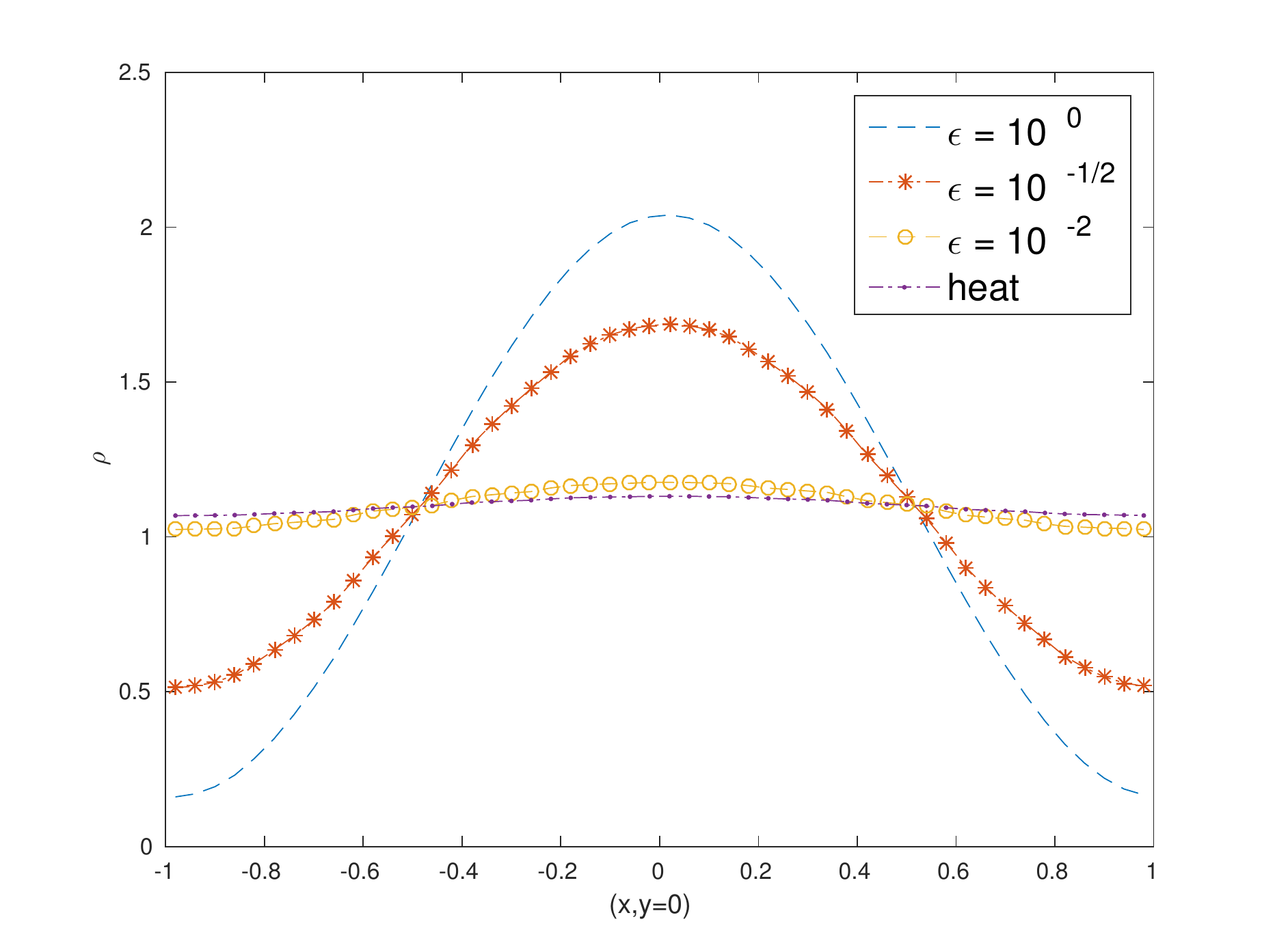}
\caption{The top three figures show the error of the solution to the transport equation and the solution to the limiting heat equation with $\varepsilon = 10^{0}, 10^{-1/2}$ and $10^{-2}$ respectively. It is obvious that as $\varepsilon$ converges to zero, the error shrinks as well. The bottom two figures show the intersection at $x=0$ and $y=0$ respectively.}\label{fig:2D_bench_AP_comparison}
\end{figure}

\bibliographystyle{abbrv}
\bibliography{transport}

\end{document}